\documentclass[11pt,reqno]{article}

\usepackage[T1]{fontenc}
\usepackage{lmodern}
\usepackage{microtype}
\usepackage{amsmath,amssymb,amsthm,mathtools}
\usepackage{bm}
\usepackage{booktabs}
\usepackage{enumitem}
\usepackage[margin=1.18in]{geometry}
\usepackage[nocompress]{cite}
\usepackage[colorlinks=true,linkcolor=blue,citecolor=blue,urlcolor=blue]{hyperref}

\allowdisplaybreaks
\numberwithin{equation}{section}
\newtheorem{theorem}{Theorem}[section]
\newtheorem{proposition}[theorem]{Proposition}
\newtheorem{lemma}[theorem]{Lemma}
\newtheorem{corollary}[theorem]{Corollary}

\newtheorem{question}[theorem]{Question}
\theoremstyle{remark}
\newtheorem{remark}[theorem]{Remark}

\newcommand{\R}{\mathbb{R}}

\newcommand{\loc}{\mathrm{loc}}

\newcommand{\diam}{\operatorname{diam}}
\newcommand{\ind}{\operatorname{ind}}
\newcommand{\cof}{\operatorname{cof}}

\newcommand{\cM}{\mathcal{M}}
\newcommand{\cB}{\mathcal{B}}
\newcommand{\cS}{\mathcal{S}}

\newcommand{\Sph}{\mathbb{S}}

\title{\bf Critical Sobolev thresholds for openness and discreteness
of Monge--Amp\`ere gradient mappings}
\author{Deguang Zhong\\
\small Institute of Applied Mathematics, Shenzhen Polytechnic University,\\[-1mm]
\small Shenzhen 518055, People's Republic of China\\
Email: \texttt{huachengzhon@163.com}}
\date{}

\begin{document}

\maketitle

\begin{abstract}
Let $\Omega\subset\R^n$ be a domain and let
$u\in W^{2,n}_{\loc}(\Omega)$ satisfy
\[
        \det D^2u\geq\delta>0\qquad\text{a.e. in }\Omega.
\]
Guerra and Tione asked whether the gradient mapping $Du$ must be open and
discrete.  We give a negative answer in every dimension $n\geq4$, in a form
that isolates the sharp regularity mechanism.  First, a local Pogorelov model
is chosen to solve the exact equation $\det D^2u=1$ a.e.  It is convex,
belongs to $C^{1,1-2/n}\cap W^{2,n}_{\loc}$, and its Hessian is positive
definite off a line.  Nevertheless, its gradient collapses that line to one
point.  In fact the gradient is neither open nor discrete, its branch set is
exactly the collapsed line, and its Jacobian equals one a.e.

We then develop a $k$-dimensional version of the construction.  If
$1\leq k<n/2$, there are convex potentials with a $k$-dimensional flat
contact set, uniformly positive and bounded Hessian determinant, and exact
critical exponent
\[
        p_{n,k}=\frac{n(n-k)}{2k}.
\]
Their Hessians belong to $L^p_{\loc}$ precisely for $p<p_{n,k}$, lie in the
weak endpoint space $L^{p_{n,k},\infty}_{\loc}$, and fail to belong to any
finite-index Lorentz endpoint.  This matches the critical minimum-set theorem
of Collins and Mooney.  At the regularity required in the question, the
construction produces branch sets of every integer dimension $k<n/3$.
We also give a topological reformulation of the problem.  In the convex
branch, gradient fibers are exactly contact sets with supporting affine
functions, and openness and discreteness are equivalent to strict convexity.
For a general gradient under the hypotheses above, critical Sobolev mapping
theory already supplies continuity and sense preservation.  The question
asks whether the monotone factor in the Eilenberg--Whyburn monotone--light
factorization is trivial.
We also determine the exact integrability of the distortion of $Du$, describe
the cuspidal geometry of its image, and explain why the same mechanism has a
logarithmic obstruction in dimension three.  An explicit three-dimensional
saddle example shows that the obstruction persists without convexity, and
we exclude the full family of fixed power-separated profiles at the endpoint.
Thus the three-dimensional
nonconvex endpoint remains separate from, and is not settled by, the present
counterexamples.
\end{abstract}



\section{Introduction}

The relation between an analytic nondegeneracy condition and the topology of
a mapping is one of the organizing principles of nonlinear analysis.  At the
$C^1$ level, the inverse function theorem says that a map
$f\colon\Omega\to\R^n$ is a local diffeomorphism near $x_0$ whenever
$\det Df(x_0)\ne0$.  At Sobolev regularity the pointwise theorem is no longer
available.  A positive Jacobian almost everywhere does not, by itself, prevent
the map from folding, branching, or collapsing a null set.  The theory of
quasiregular mappings and mappings of finite distortion quantifies how much
additional control on $Df$ is sufficient to recover openness and discreteness;
see Reshetnyak~\cite{R1989}, Rickman~\cite{R1993}, Iwaniec and
Martin~\cite{IM2001}, and Hencl and Koskela~\cite{HK2014}.

Gradient mappings have an additional differential constraint:
\[
                         Df=D^2u\in\operatorname{Sym}(n)
                         \qquad\text{a.e.}
\]
This symmetry is a substantial source of rigidity.  It is central, for
example, in the study of quasiregular gradient mappings, differential
inclusions in the symmetric matrices, and the Monge--Amp\`ere equation.  A
recent theorem of Guerra and Tione~\cite{GT2025} establishes the almost
everywhere constancy of the index of a uniformly nondegenerate Hessian at
$W^{2,\infty}_{\loc}$ regularity, thereby resolving a conjecture of
\v Sver\'ak~\cite{S1992}.  Their more general finite-distortion theorem
uses both the topology of open discrete mappings and critical groups of the
potential.

For $u\in W^{2,n}_{\loc}$ satisfying
\begin{equation}\label{eq:intro-det-lower}
      \det D^2u\geq\delta>0\qquad\text{a.e.},
\end{equation}
the distortion of $f=Du$ obeys only
\[
      K_f=\frac{|D^2u|^n}{\det D^2u}
      \leq \delta^{-1}|D^2u|^n\in L^1_{\loc}.
\]
This is below the familiar higher-dimensional openness threshold for general
mappings of finite distortion.  It is therefore natural to ask whether the
gradient constraint compensates for the missing distortion integrability.
Guerra and Tione formulated precisely this question.

\begin{question}[Guerra--Tione, Question 5.5]\label{que1}
Let $\Omega\subset\R^n$ be a domain and let
$u\in W^{2,n}_{\loc}(\Omega)$ satisfy~\eqref{eq:intro-det-lower}.  Must
$Du\colon\Omega\to\R^n$ be open and discrete?
\end{question}

The basic negative answer in dimensions $n\geq4$ already appeared in the
author's preliminary note~\cite{Z2026}.  The present expanded treatment
identifies the precise Sobolev scaling, gives an exact-density profile and a
family with higher-dimensional flat sets, and analyzes the endpoint spaces
and gradient topology.  The mechanism is the
classical Pogorelov singularity for the real Monge--Amp\`ere equation.  The
construction itself belongs to the standard repertoire of Monge--Amp\`ere
theory; see Pogorelov~\cite{P1964}, Guti\'errez~\cite{G2001},
Figalli~\cite{F2017}, and Mooney~\cite{M2018}.  In particular,
$k$-dimensional flat-set analogues were already discussed by
Caffarelli~\cite{C1993} and explicitly noted by Collins and
Mooney~\cite[Remark 1.4]{CM2017}.  What matters here is
its topological interpretation as a critical Sobolev gradient map: a null
affine set is collapsed even though the Jacobian is bounded away from zero
almost everywhere.

There is a useful topological formulation behind this collapse.  In the
convex branch, a fiber $(Du)^{-1}(p)$ is precisely the contact set of the graph
of $u$ with a supporting affine function of slope $p$.  Since such a fiber is
convex, discreteness forces it to be a singleton.  In the absence of
convexity, the same fiber is the critical set of the tilted potential
$u-p\cdot x$.  Thus Question~\ref{que1} asks whether the endpoint analytic
hypotheses exclude critical continua.  The critical Sobolev theorem of
Vodop'yanov and Gol'dshtein (in the formulation
of~\cite{GHP2019}) supplies a continuous representative of $Du$; the
positive-degree theorem of Li and Schikorra~\cite{LS2021} supplies its
sense-preserving property.  The remaining obstruction can be encoded by the
Eilenberg--Whyburn monotone--light quotient: the
monotone factor collapses the connected components of gradient fibers, and
the desired conclusion requires this factor to be trivial.  We make these
statements precise in Section~\ref{sec:topological-reformulation}.

Our first theorem records an exact Monge--Amp\`ere equation, derived from the
classical Pogorelov profile equation, and identifies the full branch set.

\begin{theorem}[Exact-density counterexample]\label{thm1}
Let $n\geq4$.  There are a bounded convex cylinder
$\Omega\subset\R^{n-1}\times\R$, a nondegenerate line segment
\[
       \Sigma=\Omega\cap(\{0\}\times\R),
\]
and a convex function
\[
       u\in C^{1,1-2/n}_{\loc}(\Omega)\cap W^{2,n}_{\loc}(\Omega)
\]
with the following properties:
\begin{enumerate}[label=\textup{(\roman*)},leftmargin=2.2em]
\item $u$ is smooth in $\Omega\setminus\Sigma$ and
      $D^2u>0$ there;
\item $\det D^2u=1$ a.e. in $\Omega$, and $u$ is an Aleksandrov solution
      of $\cM_u=\mathcal L^n$;
\item for $F=Du$ one has $J_F=1$ a.e. and
      \[
            F^{-1}(0)=\Sigma;
      \]
\item $F$ is neither open nor discrete, and its branch set is exactly
      $B_F=\Sigma$.
\end{enumerate}
In fact $u\in W^{2,p}_{\loc}(\Omega)$ for every
$p<n(n-1)/2$, whereas $u\notin W^{2,n(n-1)/2}_{\loc}(\Omega)$.
\end{theorem}

The exact equation in Theorem~\ref{thm1} is obtained by solving a regular
one-dimensional initial-value problem for the profile of the Pogorelov ansatz.
For the higher-dimensional hierarchy, an explicit exponential profile is more
convenient.  Write $n=m+k$, split points as $(z,y)\in\R^m\times\R^k$, and let
$E=\{0\}\times\R^k$.

\begin{theorem}[Flat-set hierarchy]\label{thm2}
Let $n\geq3$ and let $1\leq k<n/2$.  Set
\begin{equation}\label{eq:intro-alpha-pstar}
       m=n-k,\qquad
       \alpha=2-\frac{2k}{n},\qquad
       p_{n,k}=\frac{n(n-k)}{2k}.
\end{equation}
There exist $a,\rho>0$ and a convex cylinder
\[
       \Omega=B^m_1\times B^k_\rho
\]
such that
\begin{equation}\label{eq:intro-explicit-u}
       u(z,y)=|z|^\alpha e^{a|y|^2}
\end{equation}
has the following properties:
\begin{enumerate}[label=\textup{(\roman*)},leftmargin=2.2em]
\item $u\in C^{1,\alpha-1}_{\loc}(\Omega)$, and this H\"older exponent is
      optimal along $E$;
\item $D^2u>0$ in $\Omega\setminus E$, and for constants
      $0<c\leq C<\infty$,
      \[
             c\leq\det D^2u\leq C
             \qquad\text{a.e. in }\Omega;
      \]
\item locally across $E$,
      \[
      \begin{split}
           &D^2u\in L^p \quad\Longleftrightarrow\quad p<p_{n,k},\\
           &D^2u\in L^{p_{n,k},\infty}
             \setminus\bigcup_{1\leq q<\infty}L^{p_{n,k},q};
      \end{split}
      \]
\item if $F=Du$, then $F^{-1}(0)=E\cap\Omega$, the branch set of $F$ is
      $E\cap\Omega$, and $F$ is neither open nor discrete.
\end{enumerate}
\end{theorem}

The number $p_{n,k}$ is not an artifact of the calculation.  Collins and
Mooney~\cite{CM2017} proved that a convex solution of
$\det D^2u\geq1$ in $W^{2,p_{n,k}}$ cannot agree with a tangent plane on a
set of dimension $k$.  Theorem~\ref{thm2} realizes the complementary strict
range $p<p_{n,k}$ and also records the weak endpoint.  Thus the construction
and the rigidity theorem meet at exactly the same exponent.

At $p=n$, the condition $p<p_{n,k}$ is equivalent to $k<n/3$.  We therefore
obtain counterexamples whose branch set has any integer dimension below one
third of the ambient dimension.

\begin{corollary}[Branch sets at the Guerra--Tione exponent]\label{cor1}
If $n\geq4$ and $1\leq k<n/3$, there is a convex
$u\in C^1(\Omega)\cap W^{2,n}_{\loc}(\Omega)$ such that
\[
       0<c\leq\det D^2u\leq C<\infty\qquad\text{a.e.},
\]
while the branch set of $Du$ is a $k$-dimensional affine disk.  In
particular, $Du$ is neither open nor discrete.
\end{corollary}

There are two distinct critical quantities in the construction.  The Sobolev
threshold $p_{n,k}$ measures the integrability of the Hessian.  The distortion
threshold of $F=Du$ is
\begin{equation}\label{eq:intro-qstar}
                  q_{n,k}=\frac{n-k}{2k}.
\end{equation}
Indeed $K_F\asymp |z|^{-2k}$, and hence
$K_F\in L^q_{\loc}$ exactly for $q<q_{n,k}$, with weak endpoint at
$q=q_{n,k}$.  This makes transparent why the general finite-distortion
openness theorem cannot be applied.

The dimensional endpoint deserves emphasis.  For a line singularity,
$p_{n,1}=n(n-1)/2$.  Thus $p_{3,1}=3$, and the three-dimensional
Pogorelov Hessian has the logarithmic divergence
\[
                         \int_0^\rho\frac{dr}{r}=\infty.
\]
The counterexample therefore belongs to $W^{2,q}_{\loc}$ for every $q<3$
but not to $W^{2,3}_{\loc}$.  In the convex three-dimensional branch, the
endpoint theorem of Collins and Mooney yields strict convexity, and
Lemma~\ref{lem2} gives a continuous gradient, which is locally injective.
The complete nonconvex three-dimensional
version of Question~\ref{que1} is not resolved by this argument.  We keep
this distinction explicit throughout.

The endpoint failure of the ansatz is not confined to the convex branch:
Proposition~\ref{prop18} gives an explicit three-dimensional saddle potential
with positive Hessian determinant and a collapsed line, whose Hessian lies in
every $L^p$ below $3$ but not in $L^3$.  Proposition~\ref{prop19} shows that
no fixed power-separated profile of either Hessian signature can reach the
strong endpoint while retaining the determinant lower bound.  Under a
stationary nondegenerate profile, Proposition~\ref{prop20} rules out every
radial replacement of that power, not just logarithmic ones.

\subsection*{Historical and conceptual placement}

The present problem sits at an intersection of three theories whose critical
exponents should not be conflated.  In elliptic Monge--Amp\`ere theory,
convexity turns the determinant into a positive nonlinear measure and makes
sections of the graph the natural geometric objects.  Caffarelli's
localization theorem~\cite{C1990a} identifies strict convexity as the
gateway to interior regularity.  His interior $W^{2,p}$ estimates
\cite{C1990b} apply when the right-hand side satisfies the appropriate
continuity hypotheses; a two-sided measurable density bound alone does not
give arbitrary $p$.  Urbas~\cite{U1988} obtained strict convexity and
regularity above the line-contact threshold.  Caffarelli's degeneracy
analysis~\cite{C1993} gave the $n/2$ restriction on affine contact
sets and higher-dimensional variants of Pogorelov's example.  Collins and
Mooney~\cite{CM2017} proved the sharp strong Sobolev exclusion at the
critical exponent and noted that the corresponding examples were known.
Mooney~\cite{M2015} proved partial strict convexity outside a small
singular set.  For strictly convex solutions with bounded positive density,
De Philippis and Figalli~\cite{DF2013} proved $W^{2,1}$
regularity, and De Philippis, Figalli, and Savin~\cite{DFS2013} obtained
interior $W^{2,1+\varepsilon}$ estimates.  Pogorelov's singularity supplies the
sharp obstruction to unconditional interior smoothness in dimension at least
three~\cite{P1964,M2018}.

The second theory concerns convex gradients as mappings.  A convex potential
makes $Du$ monotone, but monotonicity alone does not imply strict monotonicity
or injectivity.  Ball's theorem~\cite{B1980} is a striking converse mechanism:
under local injectivity of the gradient, the presence of one supporting
hyperplane can force strict convexity.  Quasiconformal Jacobian questions for
mappings with convex potentials were studied, among others, by Kovalev and
Maldonado~\cite{KM2005}.  The examples below show the reverse
logical direction at low regularity: positivity of the Hessian determinant,
even together with a.e. positive definiteness, need not give local
injectivity of the gradient on the exceptional set.

The third theory is Sobolev mapping topology.  In the plane,
Iwaniec and \v Sver\'ak~\cite{IS1993} established openness and discreteness
under integrable distortion.  In higher dimensions, Manfredi and
Villamor~\cite{MV1998} proved a sufficient condition with distortion in
$L^q$, $q>n-1$.  Hencl and Koskela~\cite{HK2005} gave additional topological
hypotheses under which borderline integrability suffices, while Hencl and
Rajala~\cite{HR2013} constructed a Lipschitz finite-distortion mapping with
$K\in L^{n-1}$ that collapses a segment.  Guo, Hencl, and
Tengvall~\cite{GHT2020} studied quantitative restrictions on branch sets.
The present gradient examples have the extra symmetry $DF=D^2u$, but their
distortion is below that sufficient range.  Membership in $W^{1,n}$ is
scale invariant for maps $\R^n\to\R^n$; it does not alone settle mapping
topology.  The determinant is a product of singular values and can remain
bounded while some grow and others tend to zero.  Sobolev conventions and
fine-property background are recorded in~\cite{AF2003,EG2015}.
Here convexity gives the model, critical integrability permits its flat
fiber when $n\geq4$, and mapping topology detects the collapse invisible to
the almost-everywhere determinant.

\subsection*{Organization of the paper}
Section~\ref{sec:background} fixes the analytic and topological conventions.
Section~\ref{sec:topological-reformulation} identifies convex gradient fibers
with contact sets, gives the monotone--light reformulation, and derives the
flat-collapse dimension budget.
Section~\ref{sec:ansatz} develops the general flat-set ansatz and isolates its
scaling.  Section~\ref{sec:explicit} proves positivity and the determinant
bounds for the exponential family.  Section~\ref{sec:regularity} establishes
the sharp Sobolev, Lorentz, Orlicz, and H\"older properties.  The exact-density
line model is constructed in Section~\ref{sec:exact}.  The topology and image
geometry of the gradient are analyzed in Section~\ref{sec:topology}, and the
finite-distortion content in Section~\ref{sec:distortion}.  The relation with
critical strict-convexity theory and the low-dimensional conclusions are
discussed in Sections~\ref{sec:rigidity} and~\ref{sec:dimensions}.  Two
appendices supply weak-derivative and matrix calculations used in the text.

\section{Analytic and topological preliminaries}\label{sec:background}

\subsection{Convex potentials and Monge--Amp\`ere measure}

Let $\Omega\subset\R^n$ be open and convex.  A finite convex function
$u\colon\Omega\to\R$ is locally Lipschitz.  Its subdifferential at $x$ is
\[
  \partial u(x)=\bigl\{p\in\R^n:
  u(y)\geq u(x)+p\cdot(y-x)\text{ for all }y\in\Omega\bigr\}.
\]
An affine function $x\mapsto u(x_0)+p\cdot(x-x_0)$ with
$p\in\partial u(x_0)$ is a \emph{supporting affine function}; its contact
set consists of the points where equality holds.  A convex function is
\emph{strictly convex} if it is not affine on any nondegenerate segment.
These notions and the continuity of the gradient of an everywhere
differentiable convex function are standard in convex analysis
\cite{R1970}.
For $h>0$, the corresponding \emph{section} is
\[
 S_h(x_0,p)=\{x\in\Omega:u(x)<u(x_0)+p\cdot(x-x_0)+h\};
\]
sections are the basic localization sets in Monge--Amp\`ere regularity
\cite{C1990a,CM2017}.
For a Borel set $A\subset\Omega$, put
\[
             \partial u(A)=\bigcup_{x\in A}\partial u(x).
\]
The Aleksandrov Monge--Amp\`ere measure is
\begin{equation}\label{eq:MAmeasure}
                       \cM_u(A)=|\partial u(A)|.
\end{equation}
Here and below $|\cdot|$ denotes the appropriate Lebesgue measure.  Standard
facts about~\eqref{eq:MAmeasure} can be found in
Pogorelov~\cite{P1964}, Guti\'errez~\cite{G2001}, and
Figalli~\cite{F2017}.

If $u\in C^2$, is convex, and is strictly convex on a neighborhood, the change
of variables $p=Du(x)$ gives
\[
                      d\cM_u=\det D^2u\,dx.
\]
For a general convex function, the absolutely continuous part of $\cM_u$ has
density $\det D^2u$ at points of twice differentiability, but a singular part
may occur.  In our constructions $u$ is $C^1$, smooth off an affine set $E$,
and $Du(E)=\{0\}$.  Consequently $\cM_u(E)=0$, and no singular Monge--Amp\`ere
mass is hidden on the collapsed set.

Throughout, $W^{2,p}_{\loc}$ refers to distributional second derivatives in
$L^p_{\loc}$, as in~\cite{AF2003,EG2015}.  When $D^2u$ is
nonsingular and symmetric, $\ind(D^2u)$ is the number of its negative
eigenvalues, counted with multiplicity~\cite{S1992,GT2025}.

We shall repeatedly use the following elementary criterion.

\begin{lemma}\label{lem1}
Let $u\in W^{2,1}_{\loc}(\Omega)$ on a convex domain.  If
$D^2u(x)$ is positive semidefinite for a.e. $x$, then $u$ agrees a.e. with a
convex function.  If in addition $u\in C^1(\Omega)$, the given representative
is convex.
\end{lemma}

\begin{proof}
For every unit vector $e$, the distribution
$\partial_{ee}u=e^TD^2u\,e$ is nonnegative.  After mollification on a compactly
contained convex subdomain, $D^2(u*\eta_\varepsilon)$ is positive semidefinite,
so $u*\eta_\varepsilon$ is convex.  Passing locally uniformly at Lebesgue
points, or in $L^1_{\loc}$ followed by the standard closedness of the cone of
convex functions, gives a convex representative.  Continuity identifies it
with the original function.  See also Rockafellar~\cite{R1970}.
\end{proof}

\begin{lemma}[Critical differentiability of convex potentials]
\label{lem2}
Let $n\geq2$ and let $u$ be convex and belong to $W^{2,n}_{\loc}(\Omega)$
on an open convex set.  Then $u\in C^1(\Omega)$.
\end{lemma}

\begin{proof}
Suppose instead that $\partial u(x_0)$ contains two distinct vectors $p$
and $q$.  After restricting to a ball $B_R(x_0)\Subset\Omega$, put
$w(x)=u(x)-u(x_0)-p\cdot(x-x_0)$.  Then $w\geq0$, $w(x_0)=0$, and the
supporting inequality for $q$ gives
\[
 \sup_{|x-x_0|=r}w(x)\geq r|q-p|\qquad(0<r<R).
\]
The critical annular estimate for nonnegative convex functions
\cite[Lemma 2.2, with growth exponent $1$]{CM2017}, applied after
normalizing $w$ and rescaling $B_R(x_0)$, implies
\[
 \int_{B_{R/2}(x_0)\setminus B_r(x_0)}(\Delta w)^n\,dx
 \geq c\log\frac{R}{r}\qquad(0<r<R/4).
\]
This contradicts $\Delta w=\Delta u\in L^n(B_{R/2}(x_0))$.
Hence $\partial u(x)$ is a singleton at every $x$.  A finite convex function
that is differentiable throughout an open set has a continuous gradient
\cite{R1970}, proving the claim.
\end{proof}

\subsection{Open, discrete, and branched mappings}

A mapping $F\colon\Omega\to\R^n$ is \emph{open} if $F(U)$ is open whenever
$U\subset\Omega$ is open.  It is \emph{discrete} if every fiber $F^{-1}(y)$ is
a discrete subset of $\Omega$.  For a continuous map, following standard
terminology~\cite{GHT2020,R1993}, we define its branch set by
\begin{equation}\label{eq:branch-set}
 B_F=\{x\in\Omega:F\text{ is not a local homeomorphism at }x\}.
\end{equation}
If $F$ is smooth and $J_F(x)\ne0$, then $x\notin B_F$ by the inverse function
theorem.  The difficulty in Question~\ref{que1} is that $J_F>0$ is known
only a.e.; the exceptional null set can carry all of the branching.

For the later topological reduction, \emph{sense preserving} has its usual
degree-theoretic meaning: local Brouwer degrees, when defined, have the
positive orientation (see~\cite{TY1962,LS2021}).  At the critical Sobolev
regularity assumed here, strict positivity of the a.e. Jacobian implies
this degree property (Corollary~\ref{cor4}).  By a \emph{continuum} we mean a compact connected
set with more than one point when qualified as nondegenerate
\cite{W1942}.

We use the Hilbert--Schmidt norm for matrices.  A Sobolev mapping
$F\in W^{1,n}_{\loc}(\Omega,\R^n)$ with $J_F\geq0$ a.e. is a mapping of finite
distortion if
\[
                    |DF|^n\leq K(x)J_F
\]
for an a.e. finite measurable function $K$.  Its outer distortion is
\begin{equation}\label{eq:outer-distortion}
K_F(x)=
\begin{cases}
|DF(x)|^n/J_F(x),&J_F(x)>0,\\
1,&DF(x)=0.
\end{cases}
\end{equation}
The precise representative and topological theory are treated systematically
in~\cite{HK2014,IM2001}.  We shall not need a general theorem to
prove non-openness: the image of our mappings can be computed directly.

\subsection{Lorentz endpoint notation}

For a measurable function $g$ on a finite-measure set $U$, let
\[
             \mu_g(\lambda)=|\{x\in U:|g(x)|>\lambda\}|.
\]
For $0<p<\infty$, the weak Lebesgue space $L^{p,\infty}(U)$ consists of
functions satisfying
\[
             \sup_{\lambda>0}\lambda\mu_g(\lambda)^{1/p}<\infty.
\]
For $0<q<\infty$, the Lorentz spaces $L^{p,q}$ can be defined through the decreasing
rearrangement $g^*$ by
\[
 \|g\|_{L^{p,q}}^q
   =\int_0^{|U|}\bigl(s^{1/p}g^*(s)\bigr)^q\frac{ds}{s}.
\]
We refer to Bennett and Sharpley~\cite{BS1988}.  The model
$|z|^{-\gamma}$ in $\R^m$ has distribution
$\mu(\lambda)\asymp\lambda^{-m/\gamma}$; hence it lies in
$L^{m/\gamma,\infty}$ but in no $L^{m/\gamma,q}$ with finite $q$.
For $p<1$ these expressions are quasi-norms; the distributional conclusions
below remain valid without a Banach-space assertion.  Our logarithmic
integrals in Section~\ref{sec:regularity} are explicit Orlicz-type modulars;
the corresponding general terminology is also discussed by Collins and
Mooney~\cite[Section 2]{CM2017}.

\section{Contact sets and the monotone--light quotient}
\label{sec:topological-reformulation}

The obstruction in Question~\ref{que1} is concentrated on sets that are
invisible to the almost-everywhere Jacobian.  This section separates that
topological obstruction from the analytic estimates.  In the convex branch
the separation is exact: a gradient fiber is a contact set, and discreteness
of the gradient is equivalent to strict convexity.  Without convexity, the
appropriate object is the monotone factor obtained by collapsing the
connected components of the fibers.

\subsection{Gradient fibers as contact sets}

Let $\Omega\subset\R^n$ be open and convex and let
$u\in C^1(\Omega)$ be convex.  For $p\in\R^n$ define the tilted potential
\begin{equation}\label{eq:tilted-potential}
                         u_p(x)=u(x)-p\cdot x.
\end{equation}
The supporting-plane inequality gives the fundamental fiber identity
\begin{equation}\label{eq:fiber-argmin}
 (Du)^{-1}(p)
 =\operatorname*{argmin}_{x\in\Omega}\bigl(u(x)-p\cdot x\bigr).
\end{equation}
Indeed, $Du(x)=p$ if and only if $p\in\partial u(x)$, which is equivalent to
$u_p(y)\geq u_p(x)$ for every $y\in\Omega$.  Thus every gradient fiber of a
differentiable convex potential is convex.

\begin{proposition}[Convex topological equivalence]
\label{prop1}
Let $\Omega\subset\R^n$ be open and convex and let
$u\in C^1(\Omega)$ be convex.  The following statements are equivalent:
\begin{enumerate}[label=\textup{(\roman*)},leftmargin=2.2em]
\item $u$ is strictly convex;
\item every supporting affine function has a singleton contact set;
\item $Du$ is injective;
\item $Du$ is discrete;
\item $Du$ is open and discrete.
\end{enumerate}
When these conditions hold, $Du$ is a homeomorphism from $\Omega$ onto the
open set $Du(\Omega)$.
\end{proposition}

\begin{proof}
For a differentiable convex function, strict convexity is equivalent to the
property that a supporting affine function cannot touch the graph at two
distinct points.  To see the nontrivial direction directly, if strict
convexity fails, then $u$ is affine on a nondegenerate segment.  At an
interior point of that segment the tangent affine function agrees with $u$
on the whole segment, so its contact set is not a singleton.  Conversely, a
supporting affine function touching at $x\ne y$ forces equality in the convex
inequality along $[x,y]$.

Identity~\eqref{eq:fiber-argmin} now proves the equivalence of
\textup{(i)}--\textup{(iii)}.  Since every fiber is convex, a discrete fiber
cannot contain two different points: otherwise it would contain the entire
segment joining them.  Hence \textup{(iii)} and \textup{(iv)} are equivalent.
The implication \textup{(v)}$\Rightarrow$\textup{(iv)} is immediate.  Finally,
if $Du$ is injective, invariance of domain shows that $Du(\Omega)$ is open and
that the continuous bijection $Du\colon\Omega\to Du(\Omega)$ is a
homeomorphism.  This gives \textup{(iii)}$\Rightarrow$\textup{(v)}.
\end{proof}

\begin{corollary}[Strict-convexity form of Question~\ref{que1}]
\label{cor2}
In the convex $C^1$ branch, Question~\ref{que1} is equivalent to asking
whether its analytic assumptions force every supporting affine function to
have a singleton contact set, or equivalently whether they force strict
convexity.
\end{corollary}

This equivalence makes the dimensional dichotomy transparent.  In dimension
three the endpoint theorem of Collins and Mooney rules out a nondegenerate
contact segment at $W^{2,3}$ regularity.  In dimensions $n\geq4$, the
potentials constructed below possess a supporting affine function whose
contact set is a line segment, and Proposition
\ref{prop1} then turns this analytic flatness
into failure of both openness and discreteness.

\subsection{Tilted critical continua in the nonconvex branch}

Convexity is not needed for the elementary identity
\begin{equation}\label{eq:fiber-critical-set}
       (Du)^{-1}(p)
       =\{x\in\Omega:Du_p(x)=0\}
       =\operatorname{Crit}(u_p),
\end{equation}
where $u_p$ is given by~\eqref{eq:tilted-potential}.  What is lost without
convexity is that the critical set in~\eqref{eq:fiber-critical-set} need not
be a convex minimum set.  The general problem can therefore be read as an
endpoint exclusion problem for critical continua:
\begin{equation}\label{eq:critical-continuum-question}
 \begin{gathered}
 u\in W^{2,n}_{\loc},\qquad \det D^2u\geq\delta>0,\qquad
 E\subset\operatorname{Crit}(u_p),\\
 E\text{ a nondegenerate continuum}
 \quad\Longrightarrow\quad \text{a contradiction?}
 \end{gathered}
\end{equation}

In the smooth category, a critical point with nonsingular Hessian is isolated.
At Sobolev regularity the determinant lower bound gives nondegeneracy only
almost everywhere and in an approximate differential sense.  It does not
directly control a null exceptional set.  Hence a continuum in
\eqref{eq:critical-continuum-question}, if it exists, must exploit precisely
the gap between almost-everywhere differential information and pointwise
topology.  The Pogorelov examples realize this gap in the convex branch; the
three-dimensional saddle branch asks whether the gap survives at the exact
$W^{2,3}$ endpoint.

\subsection{The Eilenberg--Whyburn factorization}

A continuous mapping $F\colon X\to Y$ is called \emph{monotone} if all of its
fibers are connected, and \emph{light} if all of its fibers are totally
disconnected.  For a continuous surjection between compact metric spaces,
the Eilenberg--Whyburn monotone--light theorem factors
\begin{equation}\label{eq:monotone-light-factorization}
                   F=\ell\circ m,
             \qquad X\xrightarrow{\ m\ }M_F
                    \xrightarrow{\ \ell\ }Y,
\end{equation}
where $m$ is monotone and $\ell$ is light; see
Eilenberg~\cite{E1934} and Whyburn~\cite{W1942}.  The middle space
$M_F$ has a canonical description.  One declares
\begin{equation}\label{eq:fiber-component-equivalence}
 x\sim_F x'
 \quad\Longleftrightarrow\quad
 x\text{ and }x'\text{ belong to the same connected component of }
 F^{-1}(F(x)),
\end{equation}
and takes the quotient $M_F=X/{\sim_F}$.  Thus the fibers of $m$ are exactly
the connected components of the fibers of $F$.

For a gradient on a domain, this construction can be applied to the
restriction of a continuous representative to each compact metric set
$K\Subset\Omega$, with target $F(K)$.  It follows directly from
\eqref{eq:fiber-component-equivalence} that
\begin{equation}\label{eq:light-trivial-monotone}
 F|_K\text{ is light}
 \quad\Longleftrightarrow\quad
 m_K\text{ is injective}.
\end{equation}
Consequently, lightness means exactly that the monotone quotient performs no
nontrivial collapse.

The relevance of lightness to Question~\ref{que1} is classical.  A
discrete mapping is light.  Conversely, the Titus--Young interiority theorem
implies that a continuous sense-preserving light mapping between oriented
$n$-manifolds is open and discrete; see Titus and Young
\cite[Corollary, p.~333]{TY1962}.  We therefore obtain the following
conditional but exact topological reduction.

\begin{proposition}
\label{prop2}
Let $u\in W^{2,n}_{\loc}(\Omega)$ and suppose that $F=Du$ has a continuous,
sense-preserving representative.  Then
\begin{equation}\label{eq:open-discrete-light-equivalence}
 F\text{ is open and discrete}
 \quad\Longleftrightarrow\quad
 F\text{ is light}.
\end{equation}
Equivalently, on every compact localization the monotone factor in
\eqref{eq:monotone-light-factorization} is injective.
\end{proposition}

\begin{proof}
An open discrete map has discrete, hence totally disconnected, fibers.  The
reverse implication is the cited Titus--Young theorem.  The last assertion is
\eqref{eq:light-trivial-monotone}.
\end{proof}

\begin{corollary}[Unconditional topological reduction under the determinant
bound]\label{cor4}
If $u\in W^{2,n}_{\loc}(\Omega)$ and $\det D^2u\geq\delta>0$ a.e., then
$F=Du$ has a unique continuous representative; this representative is
sense preserving.  In particular, it is open and discrete if and only if it
is light.
\end{corollary}

\begin{proof}
$F\in W^{1,n}_{\loc}$ and $J_F>0$ a.e.  The critical continuity theorem
of Vodop'yanov--Gol'dshtein, as recalled in~\cite[Theorem~1 and
Corollary~2]{GHP2019}, gives a continuous representative.  The
positive-degree result for continuous critical Sobolev mappings with
positive Jacobian, proved in~\cite[Proposition~1.9]{LS2021}, makes it
sense preserving.  Apply Proposition~\ref{prop2}.  Two continuous
representatives equal a.e. agree everywhere, proving uniqueness.
\end{proof}

Thus the existence of a nondegenerate connected component of a gradient
fiber is the entire remaining obstruction; no separate continuity
assumption is needed in Question~\ref{que1}.

\subsection{Cell-like collapse and the null-set paradox}

For the examples of this paper, choose a compact convex localization whose
exceptional fiber is an affine disk $E_K$.  It is contractible, indeed
convex, and the monotone map $m_K$ collapses $E_K$ to one point.  All other
fibers are singletons.  The term \emph{cell-like} refers here to a compactum
that contracts in each of its neighborhoods, as an affine disk does; see
Daverman~\cite{D1986}.  Thus the middle space is obtained from the source
by collapsing one convex disk, while the light factor is one-to-one away
from its image.  We do not invoke a cell-like approximation theorem.

At the same time,
\begin{equation}\label{eq:null-collapse-invisibility}
 |E_K|=0,
 \qquad J_{Du}=\det D^2u\geq\delta>0\quad\text{a.e.}
\end{equation}
The two statements in~\eqref{eq:null-collapse-invisibility} are compatible
because the Jacobian records infinitesimal volume behavior only at almost
every differentiability point.  It does not detect the decomposition element
$E_K$.  The obstruction is therefore not a region on which the Jacobian
vanishes; it is a nontrivial element of the fiber decomposition supported on
a null set.  In degree-theoretic language, positive local degrees control
isolated sheets, while the present obstruction is precisely a failure of
isolation.

This also separates index constancy from openness.  The index records which
connected component of the nonsingular symmetric matrices contains $D^2u(x)$
at almost every point.  The monotone factor records connected topology inside
exceptional fibers.  The former can be constant while the latter is
nontrivial, as the index-zero examples below demonstrate.

\subsection{A flat-collapse dimension budget}

The sharp exponent $p_{n,k}$ admits a topological reading.  For $p\geq1$ set
\begin{equation}\label{eq:flat-dimension-budget}
                d_{\mathrm{flat}}(n,p)
                :=\frac{n^2}{n+2p}.
\end{equation}
The inequality $p<p_{n,k}$ is equivalent to
\begin{equation}\label{eq:budget-equivalence}
                         k<d_{\mathrm{flat}}(n,p).
\end{equation}
Hence $d_{\mathrm{flat}}(n,p)$ is the exact dimension budget for affine
contact disks within the present convex class.

\begin{proposition}
\label{prop3}
Fix $p\geq1$ and an integer $1\leq k<n/2$.  Among convex potentials whose
Monge--Amp\`ere density is bounded below by a positive constant, the following
statements hold.
\begin{enumerate}[label=\textup{(\roman*)},leftmargin=2.2em]
\item If $k<d_{\mathrm{flat}}(n,p)$, the explicit family in
      Theorem~\ref{thm2} gives a potential in $W^{2,p}_{\loc}$ whose gradient
      collapses a $k$-dimensional affine disk.
\item If $k\geq d_{\mathrm{flat}}(n,p)$, no potential in
      $W^{2,p}_{\loc}$ can have a $k$-dimensional affine contact disk.
\end{enumerate}
\end{proposition}

\begin{proof}
Part~\textup{(i)} is exactly the strict range $p<p_{n,k}$ of
Theorem~\ref{thm2}.  In part~\textup{(ii)}, the budget inequality is
equivalent to $p\geq p_{n,k}$.  Local $L^p$ membership then implies local
$L^{p_{n,k}}$ membership, whereas the endpoint theorem of Collins and Mooney
\cite[Theorem 1.2]{CM2017} excludes a $k$-dimensional affine contact
disk at that exponent.
\end{proof}

At the exponent in Question~\ref{que1}, formula
\eqref{eq:flat-dimension-budget} becomes
\begin{equation}\label{eq:W2n-dimension-budget}
                         d_{\mathrm{flat}}(n,n)=\frac n3.
\end{equation}
A line in dimension three lies exactly on the forbidden endpoint
$1=d_{\mathrm{flat}}(3,3)$, whereas for $n\geq4$ it lies strictly below the
budget.  This is the topological content of the dimension jump between three
and four.  More generally, the examples realize every integer flat-fiber
dimension $k<n/3$ at $W^{2,n}$ regularity.

We emphasize that Proposition~\ref{prop3} is sharp for
affine contact disks in the convex Monge--Amp\`ere class.  It does not, by
itself, assign a sharp threshold to arbitrary fractal fibers or to nonconvex
critical continua.  The unresolved three-dimensional saddle problem can now
be stated succinctly: does the exact-gradient constraint, together with
$D^2u\in L^3$ and $\det D^2u\geq\delta$, force the monotone factor of every
continuous sense-preserving representative of $Du$ to be trivial?

\section{The flat-set ansatz and its scaling}\label{sec:ansatz}

Fix integers $n\geq3$ and $1\leq k<n/2$, and put $m=n-k$.  We split
\[
                x=(z,y)\in\R^m\times\R^k,
                \qquad r=|z|.
\]
Let $h$ be a positive $C^2$ function on a domain in $\R^k$, let
$1<\alpha<2$, and consider
\begin{equation}\label{eq:general-ansatz}
                         u(z,y)=r^\alpha h(y).
\end{equation}
The affine set
\[
                         E=\{0\}\times\R^k
\]
is the candidate contact set.  The condition $\alpha>1$ ensures that $u$ and
$Du$ extend continuously across $E$, with $Du=0$ there.

\subsection{Block structure of the Hessian}

At a point with $r>0$, write $e_r=z/r$.  In the $m-1$ directions tangent to
$\Sph^{m-1}$ in the $z$ variables, the Hessian eigenvalue is
\begin{equation}\label{eq:tangential-general}
                   \lambda_{\mathrm{tan}}
                   =\frac{u_r}{r}
                   =\alpha r^{\alpha-2}h,
\end{equation}
with multiplicity $m-1$.  On the span of $e_r$ and the $y$ variables the
Hessian is
\begin{equation}\label{eq:large-block}
 \begin{pmatrix}
 \alpha(\alpha-1)r^{\alpha-2}h
       &\alpha r^{\alpha-1}(Dh)^T\\
 \alpha r^{\alpha-1}Dh&r^\alpha D^2h
 \end{pmatrix}.
\end{equation}
Introduce the $(k+1)\times(k+1)$ profile matrix
\begin{equation}\label{eq:profile-matrix}
 \cB_\alpha[h]=
 \begin{pmatrix}
 \alpha(\alpha-1)h&\alpha(Dh)^T\\
 \alpha Dh&D^2h
 \end{pmatrix}.
\end{equation}
If
\[
   R_r=\operatorname{diag}\bigl(r^{(\alpha-2)/2},
             r^{\alpha/2}I_k\bigr),
\]
then the matrix $M(r,y)$ in~\eqref{eq:large-block} equals
$R_r\cB_\alpha[h]R_r$.  It follows that
\begin{equation}\label{eq:block-det-scaling}
 \det M(r,y)
     =r^{(k+1)\alpha-2}\det\cB_\alpha[h].
\end{equation}
Combining~\eqref{eq:tangential-general} and
\eqref{eq:block-det-scaling} gives the fundamental identity
\begin{equation}\label{eq:general-det}
 \det D^2u
 =\alpha^{m-1}h^{m-1}
 r^{n\alpha-2m}\det\cB_\alpha[h].
\end{equation}

The determinant in~\eqref{eq:general-det} is independent of the distance to
$E$ exactly when
\begin{equation}\label{eq:balanced-alpha}
                n\alpha-2m=0,
   \qquad\text{that is,}\qquad
                \alpha=\frac{2m}{n}=2-\frac{2k}{n}.
\end{equation}
This exponent belongs to $(1,2)$ precisely when $k<n/2$.  Thus both the
allowable dimension of the flat set and the regularity of the potential are
encoded by determinant scaling.

\subsection{The profile positivity condition}

Because $h>0$ and $\alpha>1$, the upper-left entry of
$\cB_\alpha[h]$ is positive.  The Schur complement gives
\begin{equation}\label{eq:profile-schur}
 \cB_\alpha[h]>0
 \quad\Longleftrightarrow\quad
 \cS_\alpha[h]\coloneqq
 D^2h-\frac{\alpha}{(\alpha-1)h}Dh\otimes Dh>0.
\end{equation}
Moreover,
\begin{equation}\label{eq:profile-det-schur}
 \det\cB_\alpha[h]
 =\alpha(\alpha-1)h\det\cS_\alpha[h].
\end{equation}
Equations~\eqref{eq:general-det}--\eqref{eq:profile-det-schur} reduce the
$n$-dimensional Monge--Amp\`ere calculation to a $k$-dimensional profile
inequality.

\begin{proposition}
\label{prop4}
Assume that $h$ is positive and smooth near $0\in\R^k$ and that
$\det\cB_\alpha[h](0)\ne0$.  If the function~\eqref{eq:general-ansatz}
satisfies
\[
          0<c\leq |\det D^2u(z,y)|\leq C<\infty
\]
for all sufficiently small $r>0$ and $y$ near zero, then necessarily
$\alpha=2-2k/n$.
\end{proposition}

\begin{proof}
At $y=0$, all factors in~\eqref{eq:general-det} except
$r^{n\alpha-2m}$ are nonzero constants.  A two-sided bound as $r\downarrow0$
is possible only if its exponent vanishes.  Equation~\eqref{eq:balanced-alpha}
then gives the conclusion.
\end{proof}

\begin{remark}\label{rem1}
Proposition~\ref{prop4} explains why an ordinary change of the
radial power cannot repair the three-dimensional endpoint.  Keeping the
determinant nondegenerate forces the radial homogeneity, and that homogeneity
forces the logarithmic Sobolev divergence.  A nonhomogeneous counterexample in
dimension three, if one exists, must use a mechanism not captured by a fixed
profile $h$.
\end{remark}

\subsection{A robust class of admissible profiles}

The exponential profile used below is only a convenient explicit member of an
open family.  This is useful for separating the topological mechanism from a
special formula.  Let $V\Subset\R^k$ be a bounded convex domain and define
\[
\mathcal A_{\alpha}(V)=\left\{h\in C^2(\overline V):
 \begin{array}{l}
 h>0\text{ on }\overline V,\\
\displaystyle D^2h-\frac{\alpha}{(\alpha-1)h}Dh\otimes Dh>0
 \text{ on }\overline V
 \end{array}\right\}.
\]
The strict inequalities show that $\mathcal A_\alpha(V)$ is open in the
$C^2(\overline V)$ topology.

\begin{proposition}\label{prop5}
Let $1\leq k<n/2$, $m=n-k$, and $\alpha=2-2k/n$.  If
$h\in\mathcal A_\alpha(V)$, then
\begin{equation}\label{eq:stable-profile-u}
                       u_h(z,y)=|z|^\alpha h(y)
\end{equation}
is convex and $C^1$ on $B^m_1\times V$, is $C^2$ and strictly convex away
from $E=\{0\}\times V$, and satisfies
\begin{equation}\label{eq:stable-density}
          0<c_h\leq\det D^2u_h\leq C_h<\infty
          \qquad\text{a.e.}
\end{equation}
All conclusions concerning the zero fiber, branch set, strong Sobolev range,
weak Lorentz endpoint, and failure of openness are unchanged under a
sufficiently small $C^2$ perturbation of $h$ inside
$\mathcal A_\alpha(V)$.
\end{proposition}

\begin{proof}
The Schur criterion~\eqref{eq:profile-schur} gives positive definiteness of
the profile matrix uniformly on $\overline V$.  Since
$n\alpha-2m=0$, formula~\eqref{eq:general-det} contains no radial factor.
Compactness of $\overline V$ then gives~\eqref{eq:stable-density}.  The weak
extension and convexity follow as in Proposition~\ref{prop7} below.
The leading transverse Hessian remains comparable to
$r^{\alpha-2}=r^{-2k/n}$, so all integrability conclusions are unchanged.
Finally,
\[
 D_zu_h=\alpha r^{\alpha-2}h(y)z,
 \qquad D_yu_h=r^\alpha Dh(y).
\]
Thus the first component of $Du_h$ vanishes exactly when $z=0$, and then the
second component also vanishes.  The entire flat set is the zero fiber and
the image still misses $(\{0\}\times\R^k)\setminus\{0\}$.  Strict
$C^2$ inequalities persist under small $C^2$ perturbations, completing the
stability assertion.
\end{proof}

\begin{remark}\label{rem2}
The negative answer to Question~\ref{que1} is therefore not tied to one
finely tuned right-hand side.  Exact equality $\det D^2u=1$ requires a profile
equation, but the determinant inequality and the topological collapse persist
on a $C^2$-open set of profiles.  In this sense the failure of openness is
structurally stable within the flat-homogeneous class.
\end{remark}

\section{An explicit \texorpdfstring{$k$}{k}-flat family}\label{sec:explicit}

We now take
\begin{equation}\label{eq:exp-profile}
                     h_a(y)=e^{a|y|^2},\qquad a>0.
\end{equation}
The elementary identities
\begin{equation}\label{eq:exp-derivatives}
 Dh_a=2ah_a y,\qquad
 D^2h_a=h_a\bigl(2aI_k+4a^2y\otimes y\bigr)
\end{equation}
give
\begin{align}
 \cS_\alpha[h_a]
 &=2ah_a\left(I_k-\frac{2a}{\alpha-1}y\otimes y\right).
 \label{eq:exp-schur}
\end{align}
Its eigenvalues are $2ah_a$ with multiplicity $k-1$ and
\[
         2ah_a\left(1-\frac{2a|y|^2}{\alpha-1}\right)
\]
in the radial $y$ direction.  In particular, if
\begin{equation}\label{eq:rho-choice}
                  0<\rho^2<\frac{\alpha-1}{2a},
\end{equation}
then $\cS_\alpha[h_a]>0$ on $B^k_\rho$.  The matrix determinant lemma and
\eqref{eq:profile-det-schur} yield
\begin{equation}\label{eq:B-exp-det}
 \det\cB_\alpha[h_a]
 =\alpha(\alpha-1)(2a)^k h_a^{k+1}
 \left(1-\frac{2a|y|^2}{\alpha-1}\right).
\end{equation}

Set $\alpha=2-2k/n$.  Since the radial power in
\eqref{eq:general-det} cancels, we obtain an entirely explicit determinant.

\begin{proposition}
\label{prop6}
Let $1\leq k<n/2$, let $m=n-k$, and define
\[
 u_{a,k}(z,y)=|z|^{2-2k/n}e^{a|y|^2}
 \]
on $\Omega=B^m_1\times B^k_\rho$, where~\eqref{eq:rho-choice} holds.  Then,
for $z\ne0$,
\begin{equation}\label{eq:explicit-det-final}
 \det D^2u_{a,k}
 =\alpha^m(\alpha-1)(2a)^k e^{na|y|^2}
 \left(1-\frac{2a|y|^2}{\alpha-1}\right).
\end{equation}
Moreover $D^2u_{a,k}>0$ on $\Omega\setminus E$.
\end{proposition}

\begin{proof}
The identity follows by substituting~\eqref{eq:B-exp-det} into
\eqref{eq:general-det} and observing that
$h_a^{m-1}h_a^{k+1}=h_a^n$.  The $m-1$ tangential eigenvalues in
\eqref{eq:tangential-general} are positive.  Equations
\eqref{eq:profile-schur} and~\eqref{eq:exp-schur} show that the remaining
block is positive definite.  Hence the full Hessian is positive definite.
\end{proof}

If, for example, we impose the slightly stronger condition
\begin{equation}\label{eq:rho-half}
                     2a\rho^2\leq\frac{\alpha-1}{2},
\end{equation}
then~\eqref{eq:explicit-det-final} gives the quantitative bounds
\begin{equation}\label{eq:explicit-two-sided}
 \frac12\alpha^m(\alpha-1)(2a)^k
 \leq \det D^2u_{a,k}
 \leq \alpha^m(\alpha-1)(2a)^k e^{na\rho^2}.
\end{equation}

\begin{proposition}\label{prop7}
The function $u_{a,k}$ extends to a $C^1$ function on $\Omega$, with
\begin{equation}\label{eq:gradient-explicit}
 \begin{split}
 D_zu_{a,k}&=\alpha r^{\alpha-2}e^{a|y|^2}z,\\
 D_yu_{a,k}&=2ar^\alpha e^{a|y|^2}y,
 \end{split}
\end{equation}
and $Du_{a,k}=0$ on $E$.  Its classical Hessian off $E$ is its weak Hessian
on $\Omega$.  Consequently $u_{a,k}$ is convex.
\end{proposition}

\begin{proof}
Because $\alpha>1$, the two quantities in~\eqref{eq:gradient-explicit} tend
to zero as $r\downarrow0$, locally uniformly in $y$.  Thus $u_{a,k}\in C^1$
and $Du_{a,k}|_E=0$.  The most singular second derivatives have size
$r^{\alpha-2}=r^{-2k/n}$, which is locally integrable in the $m$ transverse
variables because $2k/n<m$.  Integration by parts on
$\{r>\varepsilon\}$ and passage to the limit produces no boundary measure;
a detailed verification appears in Appendix~\ref{app:weak}.  Hence the
classical expressions are the weak second derivatives.  They are positive
definite a.e. by Proposition~\ref{prop6}, so
Lemma~\ref{lem1} proves convexity.
\end{proof}

\begin{remark}\label{rem3}
The function $u_{a,k}$ is smooth and locally strictly convex off $E$, and
$\partial u_{a,k}(E)=Du_{a,k}(E)=\{0\}$.  Convexity and positive
definiteness away from $E$ ensure that $Du_{a,k}$ is injective on
$\Omega\setminus E$: equal gradients at distinct points would force an
affine segment containing a subsegment off $E$.  The change-of-variables
formula and the zero measure of the subgradient image of $E$ therefore give
\[
                d\cM_{u_{a,k}}=\det D^2u_{a,k}\,dx.
\]
In particular~\eqref{eq:explicit-two-sided} holds both in the a.e. Hessian
sense and in the Aleksandrov measure sense.
\end{remark}

\subsection{Spectral compensation}

The determinant bounds conceal a strongly anisotropic spectrum.  The next
statement makes explicit how large and small eigenvalues compensate.  If
$A$ and $B$ are positive quantities, $A\asymp B$ means that their ratio is
bounded above and below on the compact subcylinder under consideration.

\begin{proposition}
\label{prop8}
Let $\lambda_1(r,y)\geq\cdots\geq\lambda_n(r,y)>0$ be the eigenvalues of
$D^2u_{a,k}$ for $r>0$.  Locally uniformly for $y\in B^k_\rho$, as
$r\downarrow0$,
\begin{equation}\label{eq:large-small-spectrum}
 \lambda_j(r,y)\asymp r^{-2k/n}
       \quad(1\leq j\leq m),
 \qquad
 \lambda_j(r,y)\asymp r^{2-2k/n}
       \quad(m+1\leq j\leq n).
\end{equation}
Consequently,
\begin{equation}\label{eq:condition-number}
 \frac{\lambda_1(D^2u_{a,k})}{\lambda_n(D^2u_{a,k})}
 \asymp r^{-2},
\end{equation}
independently of $n$ and $k$, while
\begin{equation}\label{eq:spectral-product}
 \prod_{j=1}^n\lambda_j
 \asymp
 r^{-(2k/n)m}r^{(2-2k/n)k}=1.
\end{equation}
\end{proposition}

\begin{proof}
The $m-1$ spherical eigenvalues are exactly
$\alpha h_a r^{-2k/n}$.  For the radial--profile block, recall the factorization
\[
       M(r,y)=R_r\cB_\alpha[h_a](y)R_r,
 \qquad
       R_r=\operatorname{diag}
       \bigl(r^{-k/n},r^{1-k/n}I_k\bigr).
\]
On compact profile sets, the positive matrix $\cB_\alpha[h_a]$ is bounded
above and below by positive multiples of the identity.  Hence, in the sense of
quadratic forms,
\[
 c\left(r^{-2k/n}\xi^2+r^{2-2k/n}|\eta|^2\right)
 \leq
 \left\langle M(r,y)(\xi,\eta),(\xi,\eta)\right\rangle
 \leq
 C\left(r^{-2k/n}\xi^2+r^{2-2k/n}|\eta|^2\right).
\]
The min--max principle gives one further large eigenvalue and $k$ small
eigenvalues with the orders displayed in~\eqref{eq:large-small-spectrum}.
The quotient of these two scales is $r^{-2}$, proving
\eqref{eq:condition-number}.  Finally, $m=n-k$ and
$\alpha=2(n-k)/n$ give
\[
 -\frac{2k}{n}m+\left(2-\frac{2k}{n}\right)k=0,
\]
which proves~\eqref{eq:spectral-product} and agrees with the explicit
determinant formula.
\end{proof}

The universal condition-number blow-up $r^{-2}$ is worth noting.  The ambient
dimension and the dimension of the collapsed set redistribute the number of
large and small eigenvalues, but they do not change the ratio between the
extreme scales.  The Sobolev threshold depends on the multiplicity and on the
codimension of the tube; the pointwise ellipticity ratio does not.  This is
also why a lower determinant bound alone is topologically weak: it controls
the product of the singular values but not their separation.

\section{Sharp Sobolev, Lorentz, and H\"older regularity}
\label{sec:regularity}

Let
\begin{equation}\label{eq:s-and-pstar}
          s=2-\alpha=\frac{2k}{n},
          \qquad
          p_*=\frac{m}{s}=\frac{n(n-k)}{2k}.
\end{equation}
On each compact subcylinder whose $y$-projection is contained in
$B^k_\rho$, the Hessian has the two-sided asymptotic behavior
\begin{equation}\label{eq:Hessian-asymptotic}
       c_K r^{-s}\leq |D^2u_{a,k}(z,y)|
                    \leq C_K(1+r^{-s}).
\end{equation}
The lower bound follows already from the tangential eigenvalue
$\alpha h_a r^{-s}$; the upper bound follows from the block formula
\eqref{eq:large-block}.  Thus all endpoint statements reduce to a radial
one-dimensional integral.

\begin{proposition}\label{prop9}
Let $K$ be a compact cylinder meeting $E$ in a relatively open
$k$-dimensional disk.  Then
\begin{equation}\label{eq:strong-range}
       D^2u_{a,k}\in L^p(K)
       \quad\Longleftrightarrow\quad
       1\leq p<p_*.
\end{equation}
At the endpoint,
\begin{equation}\label{eq:weak-endpoint}
 D^2u_{a,k}\in L^{p_*,\infty}(K),
 \qquad
 D^2u_{a,k}\notin L^{p_*,q}(K)\quad(1\leq q<\infty).
\end{equation}
More precisely, for all sufficiently large $\lambda$,
\begin{equation}\label{eq:distribution-asymptotic}
 c_K\lambda^{-p_*}\leq
 |\{x\in K:|D^2u_{a,k}(x)|>\lambda\}|
 \leq C_K\lambda^{-p_*}.
\end{equation}
\end{proposition}

\begin{proof}
Using polar coordinates in $\R^m$ and~\eqref{eq:Hessian-asymptotic}, the
singular part of the $L^p$ integral is comparable to
\begin{equation}\label{eq:radial-Lp}
                  \int_0^{r_0}r^{m-1-sp}\,dr.
\end{equation}
It is finite exactly when $m-sp>0$, which is $p<p_*=m/s$.  At equality the
integral is $\int_0^{r_0}dr/r$.

For large $\lambda$, the inequality $r^{-s}>c\lambda$ is equivalent, up to
fixed constants, to $r<C\lambda^{-1/s}$.  The $m$-dimensional measure of this
tube is comparable to $\lambda^{-m/s}=\lambda^{-p_*}$, proving
\eqref{eq:distribution-asymptotic}.  The Lorentz conclusions follow from the
rearrangement characterization recalled in Section~\ref{sec:background}.
\end{proof}

The logarithmic refinement of the endpoint is equally explicit.  For
$\gamma\in\R$, define the local modular
\begin{equation}\label{eq:orlicz-modular}
 \mathfrak I_\gamma(D^2u;K)
 =\int_K\frac{|D^2u|^{p_*}}
 {\bigl[\log(e+|D^2u|)\bigr]^\gamma}\,dx.
\end{equation}

\begin{proposition}\label{prop10}
For the family $u_{a,k}$ and every compact cylinder $K$ as above,
\begin{equation}\label{eq:orlicz-iff}
                 \mathfrak I_\gamma(D^2u_{a,k};K)<\infty
                 \quad\Longleftrightarrow\quad
                 \gamma>1.
\end{equation}
\end{proposition}

\begin{proof}
At $p=p_*$ the power of $r$ in~\eqref{eq:radial-Lp} is $-1$, while
$\log(e+r^{-s})\asymp\log(e/r)$.  The singular part of
\eqref{eq:orlicz-modular} is therefore comparable to
\[
           \int_0^{r_0}\frac{dr}{r[\log(e/r)]^\gamma},
\]
which converges exactly for $\gamma>1$.
\end{proof}

\subsection{The determinant cost of logarithmic damping}

At a strong endpoint it is natural to weaken the singularity by a logarithm.
The following computation shows that this strategy is incompatible with a
uniform Monge--Amp\`ere lower bound for the separated line model.  It makes
the three-dimensional obstruction quantitative rather than merely formal.

Let $n\geq3$, $\alpha=2-2/n$, and
\begin{equation}\label{eq:log-radial}
 \phi_\gamma(r)=r^\alpha L(r)^{-\gamma},
 \qquad L(r)=\log\frac er,\qquad \gamma>0.
\end{equation}
For a smooth positive profile $h(t)$, set
$v_\gamma(z,t)=\phi_\gamma(|z|)h(t)$ in
$\R^{n-1}\times\R$.  The general radial--profile determinant identity is
\begin{equation}\label{eq:general-radial-profile-det}
 \det D^2v_\gamma
 =\left(\frac{\phi_\gamma'h}{r}\right)^{n-2}
 \left(\phi_\gamma\phi_\gamma''hh''
       -(\phi_\gamma')^2(h')^2\right).
\end{equation}

\begin{proposition}\label{prop11}
Assume that $h(0)>0$, $h'(0)=0$, and $h''(0)>0$.  Let
$p_*=n(n-1)/2$.  Then, near the punctured axis,
\begin{equation}\label{eq:log-hessian-size}
                 |D^2v_\gamma|\asymp
                 r^{-2/n}L(r)^{-\gamma},
\end{equation}
and
\begin{equation}\label{eq:log-det-decay}
                 \det D^2v_\gamma(z,0)
                 \asymp L(r)^{-n\gamma}\longrightarrow0.
\end{equation}
Moreover,
\begin{equation}\label{eq:log-endpoint-improvement}
 D^2v_\gamma\in L^{p_*}_{\loc}
 \quad\Longleftrightarrow\quad
 \gamma p_*>1.
\end{equation}
Thus every logarithmic damping strong enough to improve the endpoint
integrability destroys a uniform essential lower bound for the determinant
near the slice $t=0$.
\end{proposition}

\begin{proof}
Differentiating~\eqref{eq:log-radial} gives
\begin{align}
 \phi_\gamma'(r)
 &=r^{\alpha-1}L^{-\gamma}
       \left(\alpha+\frac\gamma L\right),
 \label{eq:log-first}\\
 \phi_\gamma''(r)
 &=r^{\alpha-2}L^{-\gamma}
 \left[
 \alpha(\alpha-1)+\frac{\gamma(2\alpha-1)}L
 +\frac{\gamma(\gamma+1)}{L^2}
 \right].
 \label{eq:log-second}
\end{align}
The tangential eigenvalue $(\phi_\gamma'/r)h$ has the size on the right of
\eqref{eq:log-hessian-size}; the remaining entries have no worse growth.
This proves~\eqref{eq:log-hessian-size}.

At $t=0$ the mixed term in~\eqref{eq:general-radial-profile-det} vanishes.
The power of $r$ in the remaining term is
\[
 (n-2)(\alpha-2)+\alpha+(\alpha-2)=n\alpha-2(n-1)=0,
\]
whereas the total logarithmic power is $-n\gamma$.  The bracketed factors in
\eqref{eq:log-first}--\eqref{eq:log-second} converge to positive constants,
which proves~\eqref{eq:log-det-decay}.  Finally, the endpoint integral is
comparable to
\[
 \int_0^{r_0}
 r^{n-2-(2/n)p_*}L(r)^{-\gamma p_*}\,dr
 =\int_0^{r_0}\frac{dr}{rL(r)^{\gamma p_*}}.
\]
It converges exactly when $\gamma p_*>1$.
The determinant is continuous on the punctured cylinder.  Values arbitrarily
close to zero at $(z,0)$ therefore occur on open sets of positive measure,
so the degeneration also rules out an almost-everywhere uniform lower bound.
\end{proof}

For $n=3$, Proposition~\ref{prop11} gives the especially clear
dichotomy
\[
 \gamma>\frac13
 \quad\Longrightarrow\quad D^2v_\gamma\in L^3_{\loc},
 \qquad
 \det D^2v_\gamma(z,0)\to0.
\]
For $\gamma\leq1/3$ the logarithm is too weak to repair the $L^3$ divergence.
This exhausts the usual scalar logarithmic corrections of the fixed-profile
Pogorelov ansatz.

We next record the sharp first-derivative modulus.  Let
\begin{equation}\label{eq:beta}
                   \beta=\alpha-1=1-\frac{2k}{n}\in(0,1).
\end{equation}

\begin{proposition}[Sharp $C^{1,\beta}$ regularity]\label{prop12}
The potential $u_{a,k}$ belongs to $C^{1,\beta}_{\loc}(\Omega)$.  It does not
belong to $C^{1,\gamma}$ near any point of $E$ for any $\gamma>\beta$.
\end{proposition}

\begin{proof}
For a fixed $y$, the map $z\mapsto |z|^\alpha$ is $C^{1,\alpha-1}$, with
the standard estimate
\[
 \bigl||z|^{\alpha-2}z-|w|^{\alpha-2}w\bigr|
 \leq C_\alpha|z-w|^{\alpha-1}.
\]
Multiplication by the smooth bounded profile $h_a$ and the formulas for the
mixed $y$ derivatives give the local $C^{1,\beta}$ estimate in all variables.
For optimality, fix $y_0$ and a unit vector $e\in\R^m$.  Then
\[
 |D_zu_{a,k}(re,y_0)-D_zu_{a,k}(0,y_0)|
 =\alpha e^{a|y_0|^2}r^\beta.
\]
This is incompatible with a $C^{0,\gamma}$ bound for $Du$ when
$\gamma>\beta$.
\end{proof}

Propositions~\ref{prop9}--\ref{prop12} prove the regularity
parts of Theorem~\ref{thm2}.  They also show that the Sobolev and H\"older
thresholds are two expressions of the same anisotropic scaling.

\section{An exact Monge--Amp\`ere density}\label{sec:exact}

For the direct answer to Question~\ref{que1}, take $k=1$ and
$m=n-1$.  Write $(z,t)\in\R^{n-1}\times\R$, $r=|z|$, and set
\begin{equation}\label{eq:line-alpha}
                         \alpha=2-\frac2n.
\end{equation}
For $u=r^\alpha h(t)$, formula~\eqref{eq:general-det} becomes
\begin{equation}\label{eq:line-general-det}
 \det D^2u
 =\alpha^{n-1}h^{n-2}
     \bigl[(\alpha-1)hh''-\alpha(h')^2\bigr].
\end{equation}
We choose the profile to make the right-hand side identically one.

\begin{lemma}\label{lem3}
There exist $\tau>0$ and a positive, even, real-analytic function
$h\colon(-\tau,\tau)\to(0,\infty)$ such that
\begin{equation}\label{eq:profile-ode}
 \alpha^{n-1}h^{n-2}
 \bigl[(\alpha-1)hh''-\alpha(h')^2\bigr]=1,
\end{equation}
with
\begin{equation}\label{eq:profile-initial}
                         h(0)=1,\qquad h'(0)=0.
\end{equation}
Moreover $h''>0$ after decreasing $\tau$ if necessary.
\end{lemma}

\begin{proof}
Equation~\eqref{eq:profile-ode} can be solved for $h''$:
\begin{equation}\label{eq:normal-ode}
 h''=\frac{\alpha}{\alpha-1}\frac{(h')^2}{h}
 +\frac{\alpha^{1-n}}{\alpha-1}h^{1-n}.
\end{equation}
The right-hand side is real analytic near $(h,h')=(1,0)$.  The standard
existence and uniqueness theorem gives a unique real-analytic local solution
with~\eqref{eq:profile-initial}.  Invariance under $t\mapsto-t$ and uniqueness
make the solution even.  Finally,
\[
                   h''(0)=\frac{\alpha^{1-n}}{\alpha-1}>0,
\]
so positivity of $h$ and $h''$ persists on a sufficiently small symmetric
interval.
\end{proof}

The exact profile admits a useful first integral.  Besides clarifying the
shape of the solution, it shows that the local model is the restriction of a
canonical maximal Pogorelov profile rather than an arbitrary local ODE
solution.

\begin{proposition}\label{prop13}
Let $h$ be the solution of Lemma~\ref{lem3}.  Put
\begin{equation}\label{eq:A-constant}
 A=\frac{2\alpha}{\alpha-1},
 \qquad
 C_n=\frac{2\alpha^{1-n}}
 {\,(\alpha-1)(n-2+A)\,}.
\end{equation}
Then, on the interval where $h\geq1$,
\begin{equation}\label{eq:first-integral}
                    (h')^2=C_n
                    \left(h^A-h^{-(n-2)}\right).
\end{equation}
The solution is strictly increasing for $t>0$ and extends to a maximal even
interval $(-T_n,T_n)$ with $0<T_n<\infty$ and
\begin{equation}\label{eq:profile-blowup}
                       h(t)\longrightarrow+\infty
                       \qquad\text{as }t\uparrow T_n.
\end{equation}
In particular, every compact symmetric subinterval of $(-T_n,T_n)$ gives an
exact-density counterexample.
\end{proposition}

\begin{proof}
Regard $z=(h')^2$ as a function of $h$ on the increasing branch.  Since
$h''=\tfrac12z'(h)$, equation~\eqref{eq:profile-ode} becomes
\begin{equation}\label{eq:z-linear}
 z'-\frac{A}{h}z
 =\frac{2\alpha^{1-n}}{\alpha-1}h^{-(n-1)}.
\end{equation}
Multiplication by the integrating factor $h^{-A}$ and integration from
$h=1$, where $z=0$, give
\[
 z h^{-A}
 =\frac{2\alpha^{1-n}}{(\alpha-1)(n-2+A)}
 \left(1-h^{-(n-2+A)}\right).
\]
This is~\eqref{eq:first-integral}.  Its right-hand side is positive for
$h>1$, so $h$ is strictly increasing on the positive half-interval.

Separation of variables gives
\begin{equation}\label{eq:T-integral}
 T_n=\frac1{\sqrt{C_n}}
 \int_1^\infty
 \frac{d\eta}{\sqrt{\eta^A-\eta^{-(n-2)}}}.
\end{equation}
The integrand is locally of order $(\eta-1)^{-1/2}$ near $1$, hence
integrable there.  At infinity it is of order $\eta^{-A/2}$, and
\[
                         \frac A2=\frac\alpha{\alpha-1}>1.
\]
Thus~\eqref{eq:T-integral} is finite.  Standard ODE continuation then implies
that the only obstruction at $T_n$ is $h\to\infty$, proving
\eqref{eq:profile-blowup}.
\end{proof}

\begin{remark}\label{rem4}
The first integral also displays the anisotropy of the exact solution.  The
large eigenvalues of the Hessian are supplied by the transverse factor
$r^{\alpha-2}$, whereas the profile equation adjusts the small eigenvalue so
that their product is exactly one.  This compensation is the analytic core of
the Pogorelov singularity.
\end{remark}

\begin{proposition}\label{prop14}
Let $h$ be given by Lemma~\ref{lem3} and define
\begin{equation}\label{eq:exact-u}
                 u(z,t)=|z|^{2-2/n}h(t)
\end{equation}
on $\Omega=B^{n-1}_1\times(-\tau,\tau)$.  Then $u$ is convex,
$C^1$ on $\Omega$, smooth off
$\Sigma=\{0\}\times(-\tau,\tau)$, and
\begin{equation}\label{eq:exact-equation}
                 D^2u>0\quad\text{on }\Omega\setminus\Sigma,
                 \qquad
                 \det D^2u=1\quad\text{a.e. in }\Omega.
\end{equation}
It is also an Aleksandrov solution of
\begin{equation}\label{eq:alex-exact}
                              \cM_u=\mathcal L^n.
\end{equation}
\end{proposition}

\begin{proof}
The determinant identity follows from~\eqref{eq:line-general-det} and the
profile equation.  The tangential eigenvalues are positive.  The upper-left
entry in the radial--$t$ block is positive, and the determinant of that block
is positive by~\eqref{eq:profile-ode}; hence the block and the full Hessian are
positive definite off $\Sigma$.  The argument of
Proposition~\ref{prop7} gives $C^1$ regularity, identifies the weak
Hessian, and proves convexity.

If two points outside $\Sigma$ had the same gradient, convexity would make
$u$ affine on the segment between them.  A subsegment outside $\Sigma$
would contradict $D^2u>0$ there.  Hence $Du$ is injective off $\Sigma$,
and the classical change of variables gives Lebesgue Monge--Amp\`ere
measure on $\Omega\setminus\Sigma$.  Since $u\in C^1$ and
$Du(\Sigma)=\{0\}$, one has
$\cM_u(\Sigma)=|\partial u(\Sigma)|=0$.  Decomposing a Borel set into its
parts on and off $\Sigma$ proves~\eqref{eq:alex-exact}.
\end{proof}

The exact model has the same sharp regularity as the exponential one.  Indeed,
on compact subcylinders,
\begin{equation}\label{eq:exact-hessian-asymptotic}
                    |D^2u(z,t)|\asymp r^{-2/n}.
\end{equation}
Therefore
\begin{equation}\label{eq:exact-regularity-range}
 u\in W^{2,p}_{\loc}(\Omega)
 \quad\Longleftrightarrow\quad
 p<\frac{n(n-1)}2,
\end{equation}
with the corresponding weak Lorentz endpoint.  Since
\[
       n<\frac{n(n-1)}2\quad\Longleftrightarrow\quad n>3,
\]
the exact solution belongs to $W^{2,n}_{\loc}$ precisely in the range needed
for Theorem~\ref{thm1}, namely $n\geq4$.

\section{Fibers, branch set, and cuspidal image geometry}
\label{sec:topology}

We analyze the topology for the general explicit family; the exact line model
is identical in the relevant respects.  Put
\[
                 F=Du_{a,k}=(P,Q)\in\R^m\times\R^k.
\]
From~\eqref{eq:gradient-explicit},
\begin{equation}\label{eq:PQ}
 P=\alpha r^{\alpha-2}e^{a|y|^2}z,
 \qquad
 Q=2ar^\alpha e^{a|y|^2}y.
\end{equation}

\begin{proposition}\label{prop15}
For $F=Du_{a,k}$ one has
\begin{equation}\label{eq:zero-fiber}
                        F^{-1}(0)=E\cap\Omega
\end{equation}
and
\begin{equation}\label{eq:branch-exact}
                        B_F=E\cap\Omega.
\end{equation}
The restriction $F|_{\Omega\setminus E}$ is a local orientation-preserving
$C^1$ diffeomorphism.
\end{proposition}

\begin{proof}
If $z=0$, then~\eqref{eq:PQ} and the $C^1$ extension give $F=0$.  Conversely,
if $z\ne0$, then $P\ne0$, so $F\ne0$.  This proves~\eqref{eq:zero-fiber}.
Every neighborhood of a point of $E$ contains infinitely many points of $E$
with the same image, so $F$ is not locally injective there.  Thus
$E\subset B_F$.  Off $E$, Proposition~\ref{prop6} gives
$J_F=\det D^2u_{a,k}>0$, and the inverse function theorem shows that $F$ is a
local orientation-preserving diffeomorphism.  Hence no other branch points
exist.
\end{proof}

The failure of openness is stronger than the failure of discreteness and can
be seen without degree theory.

\begin{proposition}[Non-openness]\label{prop16}
The mapping $F=Du_{a,k}$ is not open at any point of $E$.  More precisely,
\begin{equation}\label{eq:normal-plane-missing}
 F(\Omega)\cap\bigl(\{0\}\times\R^k\bigr)=\{0\}.
\end{equation}
Consequently, for every open set $U\subset\Omega$ meeting $E$, the image
$F(U)$ is not open.
\end{proposition}

\begin{proof}
The first component $P$ in~\eqref{eq:PQ} vanishes if and only if $z=0$.
At such points $Q=0$, proving~\eqref{eq:normal-plane-missing}.  If $U$ meets
$E$, then $0\in F(U)$; however, every Euclidean neighborhood of $0$ in the
target contains nonzero points of $\{0\}\times\R^k$, none of which belongs to
$F(\Omega)$.  Thus $0$ is not an interior point of $F(U)$.
\end{proof}

There is also a quantitative version.  On a compact subcylinder with bounded
$y$, the first equation in~\eqref{eq:PQ} gives $|P|\asymp r^{\alpha-1}$,
while the second gives $|Q|\leq Cr^\alpha$.  Hence
\begin{equation}\label{eq:cusp}
             |Q|\leq C|P|^{\alpha/(\alpha-1)}
             \qquad\text{near the origin in }F(\Omega).
\end{equation}
Since $\alpha/(\alpha-1)>1$, the image approaches the origin through a cusp
around the $P$-space and has no thickness in the normal target directions.
The analytic collapse and the topological non-openness are therefore two
facets of the same anisotropic geometry.

We can sharpen the description of nontrivial fibers by convexity.  If $u$ is
differentiable and convex, $Du(x)=Du(y)$ implies that $u$ is affine on the
segment $[x,y]$.  For the present potentials, the Hessian is positive definite
away from $E$, so every affine segment is contained in $E$.  It follows that
every fiber other than the zero fiber is a singleton.  Thus all failure of
discreteness is concentrated on one explicitly identified null set.
In the language of Section~\ref{sec:topological-reformulation}, the monotone
factor is one-to-one away from $E$ and collapses precisely the affine disk
$E\cap\Omega$.

Combining Propositions~\ref{prop15} and~\ref{prop16} with the
regularity results proves Theorem~\ref{thm2}.  Applying the same argument to
Proposition~\ref{prop14} proves all the topological assertions of
Theorem~\ref{thm1}.

\section{Finite distortion and the sharp distortion profile}
\label{sec:distortion}

The examples are gradient mappings of finite distortion with an a.e. positive
Jacobian bounded from zero and infinity.  Their failure of openness is
therefore controlled entirely by the blow-up of the distortion near $E$.

Let $F=Du_{a,k}$.  From~\eqref{eq:Hessian-asymptotic} and
\eqref{eq:explicit-two-sided},
\begin{equation}\label{eq:distortion-asymptotic}
       K_F=\frac{|D^2u_{a,k}|^n}{\det D^2u_{a,k}}
       \asymp r^{-sn}=r^{-2k}
\end{equation}
on compact cylinders meeting $E$.  Define
\begin{equation}\label{eq:qstar}
                       q_* =\frac{m}{2k}=\frac{n-k}{2k}.
\end{equation}

\begin{proposition}\label{prop17}
On each compact cylinder $K$ meeting $E$ in a relatively open $k$-disk,
\begin{equation}\label{eq:K-range}
 K_F\in L^q(K)\quad\Longleftrightarrow\quad q<q_*.
\end{equation}
Moreover
\begin{equation}\label{eq:K-lorentz}
 K_F\in L^{q_*,\infty}(K)
 \setminus\bigcup_{1\leq s<\infty}L^{q_*,s}(K).
\end{equation}
\end{proposition}

\begin{proof}
By~\eqref{eq:distortion-asymptotic}, the relevant radial integral is
\[
                  \int_0^{r_0}r^{m-1-2kq}\,dr,
\]
which converges exactly when $q<m/(2k)=q_*$.  The distribution of
$r^{-2k}$ in $m$ dimensions is comparable to $\lambda^{-m/(2k)}$.
The Lorentz statement follows exactly as in Proposition~\ref{prop9}.
\end{proof}

For the line example, $q_*=(n-1)/2$.  In particular, in dimension $n=4$ the
distortion belongs locally to $L^q$ for $q<3/2$ but not to $L^{3/2}$; in every
dimension it lies far below the general sufficient condition $K_F\in L^q$
with $q>n-1$ for openness and discreteness.  Thus the examples do not
contradict the finite-distortion theory.  They show instead that symmetry of
$DF$ and the critical membership $F\in W^{1,n}$ do not bridge the entire gap
from integrable distortion to the higher-dimensional topological threshold.

For the exact-density line model, the formula is especially transparent:
\[
                        J_F=1\quad\text{a.e.},
        \qquad K_F=|D^2u|^n\asymp r^{-2}.
\]
The mapping is continuous, belongs to $W^{1,n}_{\loc}$ for $n\geq4$, is a
smooth orientation-preserving local diffeomorphism off one line, and yet is
neither open nor discrete.  The singular line has zero $n$-dimensional
measure, but it is topologically decisive.

\section{Critical rigidity and optimality}\label{sec:rigidity}

The flat-set hierarchy is best understood beside the critical rigidity
theorem of Collins and Mooney~\cite{CM2017}.  In one of its real
forms, their result says the following.  Suppose $u$ is convex,
$\det D^2u\geq1$, and the contact set with a supporting plane contains a
$k$-dimensional piece, where $0<k<n/2$.  Then
\begin{equation}\label{eq:CM-obstruction}
                       D^2u\notin L^{p_{n,k}}_{\loc},
            \qquad p_{n,k}=\frac{n(n-k)}{2k}.
\end{equation}
Their proof is not based on the explicit Pogorelov form.  It combines the
growth of convex sections with a scale-invariant lower bound on the Hessian
mass.  At the endpoint, every dyadic annulus consumes a definite amount of
$L^{p_{n,k}}$ energy, producing logarithmic divergence.

Our explicit family supplies the converse side of this statement:
for every $p<p_{n,k}$, there is a convex solution with two-sided bounded
Monge--Amp\`ere density whose contact set contains a $k$-disk and whose
Hessian lies in $L^p$.  Proposition~\ref{prop9} further identifies
the exact weak endpoint.  We record the combined threshold as a corollary.

\begin{corollary}\label{cor3}
Fix $1\leq k<n/2$.  Within the class of convex potentials with Monge--Amp\`ere
density bounded below by a positive constant, the exponent
\[
                         p_{n,k}=\frac{n(n-k)}{2k}
\]
is the sharp strong Sobolev threshold for the presence of a $k$-dimensional
affine contact disk: such examples exist for every $p<p_{n,k}$, while they do
not exist in $W^{2,p_{n,k}}_{\loc}$.
\end{corollary}

\begin{proof}
Existence below the threshold follows from Theorem~\ref{thm2}.  The endpoint
exclusion is the real Monge--Amp\`ere theorem of Collins and
Mooney~\cite[Theorem 1.2]{CM2017}.
\end{proof}

For $k=1$, the threshold is $n(n-1)/2$.  This is the familiar optimal
regularity of the Pogorelov line singularity.  For the exponent $p=n$ in
Question~\ref{que1}, the inequality $n<p_{n,k}$ reduces to $3k<n$.  Thus
the branch-set hierarchy in Corollary~\ref{cor1} is exactly the portion of
the flat-contact diagram lying below criticality.

It is important that the counterexample has index zero a.e.  There is no
oscillation between components of the nonsingular symmetric matrices:
\begin{equation}\label{eq:index-zero}
                         \ind(D^2u)=0\qquad\text{a.e.}
\end{equation}
The phenomenon therefore does not challenge the index-constancy conclusion.
Instead it shows that openness and discreteness, proposed as a possible route
to index constancy at $W^{2,n}$ regularity, fail even in the elliptic branch.
The logical distinction is essential: a negative answer to
Question~\ref{que1} does not produce a variable-index Hessian.

\section{Low dimensions and the unresolved endpoint}\label{sec:dimensions}

We summarize what the construction proves, what known theory adds, and what
remains outside its scope.

\subsection{Dimensions one and two}

In dimension one, $u''\geq\delta>0$ makes $u'$ strictly increasing, so the
gradient is open and discrete.  In dimension two, set $F=Du$.  Then
$F\in W^{1,2}_{\loc}$ and
\[
 K_F=\frac{|D^2u|^2}{\det D^2u}
 \leq\delta^{-1}|D^2u|^2\in L^1_{\loc}.
\]
The planar theorem of Iwaniec and \v Sver\'ak~\cite{IS1993} yields a
continuous, open, discrete representative for a nonconstant map of finite
distortion with $K_F\in L^1_{\loc}$; see also
\cite{AIM2009,HK2014}.  The determinant lower bound makes
$F$ nonconstant on every component.  Hence the answer to
Question~\ref{que1} is affirmative for $n=1,2$.

\subsection{The convex branch in dimension three}

Suppose $n=3$ and, in addition to the hypotheses of
Question~\ref{que1}, assume
\begin{equation}\label{eq:convex-3d}
                         D^2u\geq0\qquad\text{a.e.}
\end{equation}
Then $u$ is convex by Lemma~\ref{lem1}.  The critical line exponent is
\[
                   p_{3,1}=\frac{3(3-1)}2=3.
\]
Lemma~\ref{lem2} makes $u$ continuously differentiable, and the
endpoint theorem of Collins and Mooney~\cite[Theorem 1.2 and
Remark 1.3]{CM2017} excludes a nontrivial contact segment.  Thus $u$
is strictly convex on every convex subdomain.  Strict convexity gives local
injectivity of the continuous gradient: on a convex ball,
for $x\ne y$,
\begin{align*}
 u(y)&>u(x)+Du(x)\cdot(y-x),\\
 u(x)&>u(y)+Du(y)\cdot(x-y).
\end{align*}
Adding yields
\begin{equation}\label{eq:strict-monotonicity}
              (Du(y)-Du(x))\cdot(y-x)>0.
\end{equation}
Thus $Du(x)\ne Du(y)$.  Invariance of domain then makes $Du$ a local
homeomorphism.  This proves the affirmative conclusion for convex potentials
in dimension three without an extra continuity assumption.

\subsection{Why the Pogorelov mechanism stops at \texorpdfstring{$n=3$}{n=3}}

For a line singularity, $|D^2u|\asymp r^{-2/n}$ in $n-1$ transverse
variables.  At $n=3$,
\begin{equation}\label{eq:3d-log}
 \int_{|t|<\tau_0}\int_{|z|<\rho}|D^2u|^3\,dz\,dt
 \asymp\int_0^\rho r^{-2}r\,dr
 =\int_0^\rho\frac{dr}{r}=\infty.
\end{equation}
Nevertheless $D^2u\in L^q_{\loc}$ for every $q<3$ and
$D^2u\in L^{3,\infty}_{\loc}$.  Hence the example is an exact near-counterexample
at the endpoint.

A naive logarithmic damping illustrates the rigidity.  Formally replace
$r^{4/3}$ by
\[
                r^{4/3}\left(\log\frac er\right)^{-\gamma}.
\]
The Hessian becomes integrable in $L^3$ when $3\gamma>1$, but the determinant
acquires a factor asymptotic to
$[\log(e/r)]^{-3\gamma}$ and tends to zero.  Thus the same logarithm that
repairs the energy destroys the uniform determinant lower bound.  This is the
endpoint balance behind Propositions~\ref{prop4} and~\ref{prop11}.

\subsection{The nonconvex three-dimensional branch}

Positive determinant does not imply positive definiteness.  In dimension
three, a nonsingular symmetric matrix with positive determinant can have
index $0$ or $2$.  The convex argument treats the index-zero branch but does
not apply to a saddle Hessian of index two.  Monge--Amp\`ere comparison,
sections, and Aleksandrov estimates are elliptic tools and cease to be
available in that branch.

The following calculation tests the precise endpoint in the index-two
branch.  It shows in particular that positivity of the determinant and a
collapsed fiber coexist even when the Hessian is everywhere a saddle off
the fiber.

\begin{proposition}\label{prop18}
Let $\Omega=B^2_1\times(-\frac12,\frac12)$, write $r=|z|$ for
$z\in\R^2$, and set
\begin{equation}\label{eq:saddle-example}
             u(z,t)=r^{4/3}(t^2-1).
\end{equation}
Then $u\in C^1(\Omega)\cap W^{2,p}_{\loc}(\Omega)$ for every $p<3$,
whereas $u\notin W^{2,3}_{\loc}$ on any neighborhood meeting the axis
$E=\{z=0\}$.  Off $E$ its Hessian has index two and
\begin{equation}\label{eq:saddle-det}
       \det D^2u=\frac{32}{27}(1-t^2)(1+7t^2)
              \geq\frac{32}{27}.
\end{equation}
Moreover $(Du)^{-1}(0)=E$, its branch set equals $E$, and $Du$ is
neither open nor discrete.
\end{proposition}

\begin{proof}
Put $h(t)=t^2-1$ and $\alpha=4/3$.  In the orthonormal frame
consisting of the angular direction in $\R^2$, the radial direction, and
the $t$-direction, the Hessian for $r>0$ is
\begin{equation}\label{eq:saddle-hessian}
D^2u=\begin{pmatrix}
 \alpha h r^{\alpha-2}&0&0\\
 0&\alpha(\alpha-1)h r^{\alpha-2}&\alpha h' r^{\alpha-1}\\
 0&\alpha h' r^{\alpha-1}&h''r^\alpha
\end{pmatrix}.
\end{equation}
Thus
\[
 \det D^2u=\alpha^2r^{3\alpha-4}h
          \bigl((\alpha-1)hh''-\alpha(h')^2\bigr)
       =\frac{32}{27}(1-t^2)(1+7t^2).
\]
The last factor is at least one for $|t|<1/2$.  The angular eigenvalue is
negative because $h<0$; since the determinant is positive, the remaining
$2\times2$ block has negative determinant and hence one positive and one
negative eigenvalue.  This proves the index claim.

On a cylinder crossing $E$, the angular eigenvalue has absolute value
bounded below by $c r^{-2/3}$, while every entry in~\eqref{eq:saddle-hessian}
is bounded above by $C r^{-2/3}$.  Consequently
\[
 \int_{|z|<\rho,\,|t|<\tau}|D^2u|^p\,dz\,dt
       \asymp\int_0^\rho r^{1-2p/3}\,dr,
\]
which is finite exactly when $p<3$.  These are the distributional second
derivatives across $E$ by the integration-by-parts argument in
Appendix~\ref{app:weak}.  The first derivatives extend continuously to
zero on $E$, and for $r>0$,
\[
 Du(z,t)=\bigl(\alpha h(t)r^{\alpha-2}z,\,2t r^\alpha\bigr).
\]
As $h<0$, the first two components vanish only when $r=0$, where the third
also vanishes.  Hence $(Du)^{-1}(0)=E$.  The inverse function theorem applies
off $E$, while $Du$ is constant on $E$, so the branch set is precisely $E$.
Finally, the image intersects the target $t$-axis only at the origin; every
neighborhood of a point of $E$ therefore has an image containing the origin
but no open ball around it.  This proves non-openness.
\end{proof}

\begin{proposition}
\label{prop19}
Let $I\subset\R$ be an open interval, let $1<\alpha<2$, and let
$h\in C^2(I)$.  On $B^2_\rho\times I$, set
$v(z,t)=|z|^\alpha h(t)$.  If
$\det D^2v\geq\delta>0$ almost everywhere on this cylinder, then
$v\notin W^{2,3}_{\loc}$ on some cylinder meeting its axis.  This applies
to both the index-zero and index-two branches.
\end{proposition}

\begin{proof}
For $r>0$, the calculation in~\eqref{eq:saddle-hessian}, with $h$ left
arbitrary, gives
\begin{equation}\label{eq:saddle-power-balance}
 \det D^2v
 =\alpha^2 r^{3\alpha-4}h(t)
      \bigl((\alpha-1)h(t)h''(t)-\alpha(h'(t))^2\bigr).
\end{equation}
If $\alpha>4/3$, the right-hand side tends uniformly to zero as $r\downarrow0$
on every compact subinterval of $I$, contradicting the positive lower
bound.  Thus $\alpha\leq4/3$.  The determinant condition rules out
$h\equiv0$.  Choose a compact interval $J\Subset I$ on which $|h|\geq c>0$.
The angular eigenvalue of $D^2v$ equals $\alpha h(t)r^{\alpha-2}$, so
\[
 \int_{J}\int_{|z|<\varepsilon}|D^2v|^3\,dz\,dt
 \geq c'\int_0^\varepsilon r^{3\alpha-5}\,dr=\infty.
\]
The weak-Hessian calculation across the axis is again justified in
Appendix~\ref{app:weak}.  No convexity or fixed Hessian index was used.
\end{proof}

\begin{proposition}\label{prop20}
Let $I$ be an open interval containing $0$, let
$\phi\in C^2((0,\rho))\cap C^1([0,\rho))$ satisfy
$\phi(0)=\phi'(0)=0$ and $\phi,\phi'>0$ on $(0,\rho)$, and let
$h\in C^2(I)$ satisfy
\[
                h(0)\ne0,\qquad h'(0)=0,\qquad h''(0)>0.
\]
Set $v(z,t)=\phi(|z|)h(t)$ on $B^2_\rho\times I$.  If
$\det D^2v\geq\delta>0$ almost everywhere, then, writing
$c=\delta/[h(0)^2h''(0)]$, one has
\begin{equation}\label{eq:radial-endpoint-lower}
             \phi'(r)\geq(3cr)^{1/3}\qquad(0<r<\rho).
\end{equation}
In particular, $v\notin W^{2,3}_{\loc}$ on any neighborhood of an axis
point $(0,t)$ with $t$ sufficiently close to $0$.  Both signs of $h(0)$
are allowed.
\end{proposition}

\begin{proof}
The determinant is continuous for $r>0$, so the almost-everywhere bound
extends to every point of the punctured cylinder.  Evaluating the radial
identity~\eqref{eq:general-radial-profile-det} for $n=3$ at $t=0$ gives
\[
 \frac{h(0)^2h''(0)}r\,\phi(r)\phi'(r)\phi''(r)
                       =\det D^2v(z,0)\geq\delta.
\]
In particular $\phi''>0$.  As $\phi'(0)=0$, the increasing function
$y=\phi'$ satisfies $\phi(r)=\int_0^r y(s)\,ds\leq ry(r)$.
Consequently
\[
             c r\leq\phi(r)y(r)y'(r)
                         \leq r y(r)^2y'(r),
\]
and integration of $(y^3)'=3y^2y'$ from $0$ to $r$ proves
\eqref{eq:radial-endpoint-lower}.  Choose a compact interval
$J\Subset I$ containing $0$ with $|h(t)|\geq c_0>0$ there.  The angular
Hessian eigenvalue is $h(t)y(r)/r$, whence
\[
 \int_J\int_{|z|<\varepsilon}|D^2v|^3\,dz\,dt
 \geq C\int_0^\varepsilon r\left(\frac{y(r)}r\right)^3dr
 \geq C'\int_0^\varepsilon\frac{dr}{r}=\infty.
\]
Thus the classical Hessian off the axis is already nonintegrable, regardless
of any possible distribution supported on the axis.
\end{proof}

The logarithmic repair fails for the saddle model as well.  If
$r^{4/3}$ in~\eqref{eq:saddle-example} is replaced by
$r^{4/3}(\log(e/r))^{-\gamma}$, then its Hessian belongs to $L^3$
precisely for $\gamma>1/3$, while the determinant on the slice $t=0$
is positive and comparable to $(\log(e/r))^{-3\gamma}\to0$.
These statements follow directly from~\eqref{eq:general-radial-profile-det}
and~\eqref{eq:log-first}--\eqref{eq:log-second}, with $h(t)=t^2-1$;
continuity off the axis rules out a uniform essential determinant lower
bound in a neighborhood of $t=0$.

The two-dimensional quasiconvexity proof of Agazzi, Astala, Bruno, Cassese,
Chang, Faraco, Guerra, Kirchheim, Koski, Kristensen, Pasqualotto, Prause,
Tione, \v Sver\'ak, and Sz\'ekelyhidi~\cite{AEtAl2026} provides a different
idea: project one scalar gradient onto a convex constraint, then pair the
result with the other row through the null-Lagrangian determinant.  See also
Guerra's earlier result for symmetric $2\times2$ matrices~\cite{Guerra2026}.
For the present problem, a rowwise projection does not generally preserve
the symmetry $D(Du)=D^2u$, and an integral determinant inequality alone
does not rule out collapse of a null fiber; Proposition~\ref{prop18} already
exhibits such a fiber with a positive determinant.  In particular, an
extension of those quasiconvexity inequalities would need an additional
endpoint estimate controlling the geometry of gradient fibers.

Consequently the conclusions established here and in the cited planar theory
are
\begin{center}
\begin{tabular}{@{}cl@{}}
\toprule
$n$&Status of Question~\ref{que1}\\
\midrule
$1$&affirmative by monotonicity;\\
$2$&affirmative by planar finite-distortion theory;\\
$3$&affirmative for convex potentials; full nonconvex endpoint not settled here;\\
$n\geq4$&negative, even for convex exact Monge--Amp\`ere solutions.\\
\bottomrule
\end{tabular}
\end{center}

Any complete three-dimensional counterexample must avoid the fixed homogeneous
Pogorelov balance, while any affirmative proof must use more than convex
Monge--Amp\`ere theory.  One possible formulation is an endpoint lightness
principle: if $u\in W^{2,3}_{\loc}$, $\det D^2u\geq\delta$, and a fiber of
$Du$ contains a nondegenerate continuum, must the $L^3$ norm of $D^2u$
diverge on every neighborhood of that continuum?  The present construction
shows exactly why such a principle is true in the convex homogeneous model,
but it does not prove it for saddle Hessians.  Corollary~\ref{cor4} shows that
this is precisely the missing topological step, rather than merely a
necessary auxiliary estimate.

\subsection{An Onninen--Pankka route to the nonconvex endpoint}

The inverse-function theorem of Kovalev, Onninen, and Rajala
\cite[Theorem~1.1]{KOR2010} uses an integrable \emph{inner} distortion
\[
 K_I(F)=\frac{|\cof DF|^3}{(\det DF)^2}
\]
and an additional differential inclusion.  In contrast, our hypotheses
give only
\begin{equation}\label{eq:inner-distortion-gap}
 K_I(Du)^{1/2}\leq C\delta^{-1}|D^2u|^3\in L^1_{\loc}.
\end{equation}
This distinction is sharp at the level of matrices: for
$H_\Lambda=\operatorname{diag}(\Lambda,-\Lambda,-\Lambda^{-2})$,
$\det H_\Lambda=1$, $|H_\Lambda|^3\asymp\Lambda^3$, whereas
$|\cof H_\Lambda|^3\asymp\Lambda^6$.  Each such matrix is the
Hessian of a quadratic saddle.  Thus the cited inverse-function theorem
cannot simply be applied to $Du$.

Onninen and Pankka study distortion, growth, and geometric topology in
\cite{OP2012,OP2021}.  Their wild-involution construction has Sobolev
regularity $W^{1,p}$ for $p<2$; it warns that the exceptional geometry in
three dimensions need not admit a convenient smooth arc model.  The
monotone-deformation work of Iwaniec, Onninen, Pankka, and Radice
\cite{IOPR2021} likewise shows why almost-everywhere injectivity is too
weak a replacement for pointwise lightness.  These works motivate a
topological and modulus-based attack, but none of their theorems implies
the missing three-dimensional assertion under~\eqref{eq:intro-det-lower}.

Here is a distributional identity available at precisely the present
regularity.  It is valid for general Sobolev maps; its eventual use must
also exploit the special relation $F=Du$.

\begin{lemma}[Cofactor flux]\label{lem4}
Let $u\in W^{2,3}_{\loc}(\Omega)$, let $F=Du$ denote its continuous
representative, and put $H=D^2u$, $C=\cof H$.  For every $y\in\R^3$,
\begin{equation}\label{eq:cofactor-flux}
 \operatorname{div}\bigl(C(F-y)\bigr)=3\det H
 \quad\text{in }\mathcal D'(\Omega).
\end{equation}
\end{lemma}

\begin{proof}
The cofactor $C$ lies in $L^{3/2}_{\loc}$ and the continuous map $F$
is locally bounded.  For a smooth potential, Piola's identity gives
$\operatorname{div}C=0$ (here $C$ is symmetric), and
$C:H=3\det H$ proves~\eqref{eq:cofactor-flux}.
For the Sobolev potential, mollify on compact subdomains.  The mollified
Hessians converge in $L^3$, their cofactors in $L^{3/2}$, and their
determinants in $L^1$; the mollified gradients converge locally uniformly
because $F$ is continuous.  Passing to the limit in the smooth identity
proves the assertion.
\end{proof}

The following sufficient condition translates exclusion of a fiber
continuum into a concrete critical-energy estimate.  Its conclusion is
rigorous; establishing its hypothesis for every gradient in Question~\ref{que1}
is a separate, presently unresolved task.

\begin{proposition}[Logarithmic fiber criterion]\label{prop21}
Let $u\in W^{2,3}_{\loc}(\Omega)$ satisfy $\det D^2u\geq\delta>0$
a.e., and let $F=Du$ be its continuous representative.  Fix a ball
$B_R\Subset\Omega$, a value $y\in\R^3$, and $a>0$.  If
\begin{equation}\label{eq:log-fiber-criterion}
 \liminf_{\varepsilon\downarrow0}
 \frac{1}{\log^3(a/\varepsilon)}
 \int_{B_R\cap\{\varepsilon<|F-y|<a\}}
       \frac{|D^2u|^3}{|F-y|^3}\,dx=0,
\end{equation}
then $F^{-1}(y)\cap B_{R/2}$ contains no nondegenerate continuum.
If~\eqref{eq:log-fiber-criterion} holds for every such ball and every
$y$, then $F$ is open and discrete.
Conversely, if $E\subset F^{-1}(y)\cap B_{R/2}$ is a nondegenerate
continuum, then
\begin{equation}\label{eq:log-fiber-obstruction}
 \liminf_{\varepsilon\downarrow0}
 \frac{1}{\log^3(a/\varepsilon)}
 \int_{B_R\cap\{\varepsilon<|F-y|<a\}}
       \frac{|D^2u|^3}{|F-y|^3}\,dx
 \geq c\,\frac{\diam E}{R}>0,
\end{equation}
with an absolute constant $c$.
\end{proposition}

\begin{proof}
Suppose that a nondegenerate continuum
$E\subset F^{-1}(y)\cap B_{R/2}$ exists.  The fiber $F^{-1}(y)$ has
three-dimensional measure zero: at almost every density and approximate
differentiability point of a level set of a Sobolev map, its derivative
vanishes on that level set, contrary here to $\det DF\geq\delta$.
Choose $\zeta\in C_c^\infty(B_R)$ equal to one near $E$ and set, for
$0<\varepsilon<a$,
\[
 v_\varepsilon(x)=
 \min\!\left\{1,\max\!\left\{0,
 \frac{\log(a/|F(x)-y|)}{\log(a/\varepsilon)}\right\}\right\},
 \qquad v_\varepsilon=1\text{ on }F^{-1}(y).
\]
The chain rule gives
\[
 |Dv_\varepsilon|
 \leq\frac{|D^2u|}{|F-y|\log(a/\varepsilon)}
 \boldsymbol1_{\{\varepsilon<|F-y|<a\}}\quad\text{a.e.}
\]
For a sequence $\varepsilon_j\downarrow0$ realizing the liminf,
$\int_{B_R}|D(\zeta v_{\varepsilon_j})|^3\to0$:
the first product term is bounded by~\eqref{eq:log-fiber-criterion},
and the other tends to zero by dominated convergence, since
$v_\varepsilon\to0$ off the null fiber.  Each product is one in a
neighborhood of $E$ and has compact support in $B_R$.

This contradicts the positive relative $3$-capacity of a nondegenerate
continuum.  For completeness, let $d=\diam E>0$ and project $E$ onto
the line through two diameter points.  The projection contains an interval
of length $d$.  On each perpendicular two-plane over this interval, a
compactly supported test function equal to one near $E$ attains one
somewhere; the two-dimensional Morrey inequality for $W^{1,3}$ gives a
lower bound $c/R$ for its planar gradient energy.  Integration across
the interval yields $\operatorname{cap}_3(E;B_R)\geq cd/R>0$.
Smooth approximation of the products, using that they equal one near
$E$, justifies them as capacity tests.  Retaining the constants in the
same estimate and letting $\varepsilon\downarrow0$ gives
\eqref{eq:log-fiber-obstruction}.  The final assertion follows from
Corollary~\ref{cor4}: a nontrivial connected fiber component would contain
a nondegenerate continuum inside a smaller ball.
\end{proof}

The gradient constraint is indispensable for any proposed proof of the
estimate~\eqref{eq:log-fiber-criterion}.  Indeed, on
$\{(z,t):0<|z|<1,\ |t|<1\}\subset\R^2\times\R$, the map
\begin{equation}\label{eq:non-gradient-collapse}
 G(z,t)=\bigl(|z|^{-1/2}z,\ |z|^{1/2}t\bigr),
 \qquad G(0,t)=0,
\end{equation}
is continuous, belongs to $W^{1,3}_{\loc}$ across the axis, and has
$\det DG=\tfrac12|z|^{-1/2}\geq\tfrac12$ a.e.; nonetheless its
fiber over zero contains the entire axis.  In the cylinder,
$|DG|\asymp |z|^{-1/2}$ and $|G|\asymp |z|^{1/2}$, so its analogue of the
integral in~\eqref{eq:log-fiber-criterion} grows at least as a positive
power of $\varepsilon^{-1}$.  Since $\partial_tG_i=0$ for $i=1,2$,
whereas $\partial_{z_i}G_3$ is usually nonzero, $DG$ is not symmetric.
An endpoint proof must use this missing curl-free identity to upgrade
the cofactor/degree flux to control of critical fiber capacity; the
flux identity alone also holds for~\eqref{eq:non-gradient-collapse}.

\section{Further consequences and structural remarks}\label{sec:further}

\subsection{Affine covariance}

The examples are stable under nonsingular affine changes of variables.  If
$A\in GL(n)$, $c>0$, and
\[
                         v(x)=c\,u(Ax)+\ell(x)
\]
for an affine function $\ell$, then
\[
       D^2v(x)=cA^TD^2u(Ax)A,
       \qquad
       \det D^2v(x)=c^n(\det A)^2\det D^2u(Ax).
\]
Thus a flat branch disk can be placed on any affine $k$-plane, and the
determinant can be normalized to any prescribed positive constant.  The
addition of $\ell$ merely translates the collapsed gradient value.

\subsection{Localization to smooth domains}

Although cylinders make the formulas transparent, the conclusion is local.
One may restrict any constructed potential to a Euclidean ball compactly
contained in the cylinder and intersecting $E$.  The resulting domain is
smooth and strictly convex, the determinant bounds are unchanged, and the
gradient still collapses the affine disk $E\cap B$.  Thus the failure is not a
boundary-corner artifact.

\subsection{No hidden singular Hessian measure}

Because $u\in C^1$ and its weak Hessian belongs to $L^1_{\loc}$, the
distributional Hessian has no measure component supported on $E$.  The
positivity of $D^2u$ a.e. is therefore genuine distributional convexity, not
the cancellation of a negative absolutely continuous part by a positive
singular measure.  Likewise, in the exact model the identity
$\cM_u=\mathcal L^n$ contains no Dirac or lower-dimensional component.

\subsection{A null set can control an open-mapping property}

The determinant condition is blind to $E$, since $|E|=0$.  The topological
property is not.  All points off $E$ satisfy the hypotheses of the classical
inverse function theorem, while every point of $E$ is a branch point.  This
illustrates a general warning in borderline Sobolev topology: an a.e.
nondegeneracy statement can coexist with a large branch continuum when the
energy is just weak enough to tolerate the transverse singularity.

\subsection{Index constancy versus local invertibility}

The examples separate two questions that coincide for smooth uniformly
nondegenerate gradients but diverge at Sobolev regularity.  The index is
constant because $D^2u>0$ a.e.  Local invertibility fails because $Du$ is not
locally one-to-one on $E$.  Thus constancy of the component of
$\operatorname{Sym}(n)\setminus\{\det=0\}$ visited a.e. does not determine the
topological behavior of the continuous Sobolev representative on the
exceptional set.

\section{Proofs of the main statements}\label{sec:mainproofs}

\begin{proof}[Proof of Theorem~\ref{thm1}]
Take the exact profile of Lemma~\ref{lem3} and the potential
\eqref{eq:exact-u}.  Proposition~\ref{prop14} gives convexity,
smoothness off $\Sigma$, positive definiteness, the exact a.e. equation, and
the Aleksandrov identity.  Equation~\eqref{eq:exact-regularity-range} gives
$u\in W^{2,n}_{\loc}$ exactly when $n\geq4$ and yields the asserted sharper
$W^{2,p}$ range.  The proof of Proposition~\ref{prop12}, applied to the
positive smooth profile $h$ with $k=1$, gives $C^{1,1-2/n}$ regularity.
The argument of
Proposition~\ref{prop15} shows that $(Du)^{-1}(0)=\Sigma$ and that the
branch set equals $\Sigma$.  Finally, the missing-normal-line argument of
Proposition~\ref{prop16} proves non-openness, while the nondegenerate
zero fiber proves non-discreteness.  Since $D(Du)=D^2u$, the Jacobian of the
gradient is one a.e.
\end{proof}

\begin{proof}[Proof of Theorem~\ref{thm2}]
Choose $a>0$ and $\rho>0$ satisfying~\eqref{eq:rho-half}.  The determinant
formula and positive definiteness follow from
Proposition~\ref{prop6}; the two-sided bounds are
\eqref{eq:explicit-two-sided}.  Proposition~\ref{prop7} gives the
$C^1$ extension, the weak Hessian, and convexity.  The sharp Sobolev and
Lorentz assertions are Proposition~\ref{prop9}, and the optimal
H\"older statement is Proposition~\ref{prop12}.  The fiber, branch-set,
non-openness, and non-discreteness conclusions are Propositions
\ref{prop15} and~\ref{prop16}.
\end{proof}

\begin{proof}[Proof of Corollary~\ref{cor1}]
Theorem~\ref{thm2} gives $u\in W^{2,n}_{\loc}$ if and only if
\[
 n<p_{n,k}=\frac{n(n-k)}{2k}.
\]
After cancellation this is $3k<n$.  All other assertions are already part of
Theorem~\ref{thm2}.
\end{proof}

\appendix

\section{Weak derivatives across the flat set}\label{app:weak}

For completeness we justify the distributional calculation used above.  Let
$U=B^m_R\times V$, where $V\Subset\R^k$, and let
\[
                      v(z,y)=|z|^\alpha h(y),
       \qquad 1<\alpha<2,\quad h\in C^2(V).
\]
The classical first derivatives on $\{r>0\}$ extend continuously by zero in
the $z$-gradient and by zero in the $y$-gradient at $r=0$:
\[
 D_zv=\alpha r^{\alpha-2}h z,\qquad
 D_yv=r^\alpha Dh.
\]
The classical second derivatives satisfy
\begin{align}
 D^2_{zz}v
 &=\alpha h r^{\alpha-2}I_m
   +\alpha(\alpha-2)h r^{\alpha-4}z\otimes z,
 \label{eq:app-zz}\\
 D^2_{zy}v
 &=\alpha r^{\alpha-2}z\otimes Dh,
 \label{eq:app-zy}\\
 D^2_{yy}v
 &=r^\alpha D^2h.
 \label{eq:app-yy}
\end{align}
The largest term is bounded by $Cr^{\alpha-2}$.  Since
\[
 \int_0^{R}r^{m-1}r^{\alpha-2}\,dr<\infty
 \quad\text{whenever}\quad m+\alpha-2>0,
\]
the displayed matrices belong to $L^1_{\loc}$.  In our range $m=n-k>n/2$
and $\alpha>1$, so this condition is automatic.

Let $\varphi\in C_c^\infty(U)$ and integrate by parts on
$U_\varepsilon=U\cap\{r>\varepsilon\}$.  For a first derivative
$w=D_iv$, the boundary term on $\{r=\varepsilon\}$ is bounded by
\[
       C\,\varepsilon^{m-1}
       \sup_{r=\varepsilon}|w|.
\]
If $i$ is a $z$ direction, $|w|\leq C\varepsilon^{\alpha-1}$, and the bound
is $C\varepsilon^{m+\alpha-2}\to0$.  If $i$ is a $y$ direction, the bound is
$C\varepsilon^{m+\alpha-1}\to0$.  Hence no distribution supported on $E$
appears, and~\eqref{eq:app-zz}--\eqref{eq:app-yy} are precisely the weak
second derivatives.

The same argument shows that the pointwise estimates used in
Section~\ref{sec:regularity} are sufficient for membership in $W^{2,p}$:
there is no additional singular measure at the flat set that would be missed
by the a.e. calculation.

\section{Matrix determinant calculation in invariant form}\label{app:matrix}

We give a coordinate-free verification of~\eqref{eq:general-det}.  Let
$T_z\Sph^{m-1}$ denote the tangent space orthogonal to $e_r$.  Relative to the
orthogonal decomposition
\[
 \R^n=T_z\Sph^{m-1}\oplus\operatorname{span}\{e_r\}\oplus\R^k,
\]
the Hessian of $r^\alpha h(y)$ is
\[
 \begin{pmatrix}
 \alpha r^{\alpha-2}h I_{m-1}&0&0\\
 0&\alpha(\alpha-1)r^{\alpha-2}h
       &\alpha r^{\alpha-1}(Dh)^T\\
 0&\alpha r^{\alpha-1}Dh&r^\alpha D^2h
 \end{pmatrix}.
\]
The first block contributes
$\alpha^{m-1}r^{(m-1)(\alpha-2)}h^{m-1}$.  In the remaining block, factor
$r^{(\alpha-2)/2}$ from the radial row and column and $r^{\alpha/2}$ from
each of the $k$ profile rows and columns.  The resulting determinant factor is
\[
 r^{\alpha-2}r^{k\alpha}=r^{(k+1)\alpha-2},
\]
and the residual matrix is $\cB_\alpha[h]$.  Adding exponents gives
\[
 (m-1)(\alpha-2)+(k+1)\alpha-2
 =n\alpha-2m,
\]
which proves~\eqref{eq:general-det}.

For the exponential profile, the rank-one determinant formula
$\det(I-cy\otimes y)=1-c|y|^2$ yields
\[
 \det\cS_\alpha[h_a]
 =(2ah_a)^k\left(1-\frac{2a|y|^2}{\alpha-1}\right).
\]
Multiplication by the scalar Schur factor
$\alpha(\alpha-1)h_a$ produces~\eqref{eq:B-exp-det} and completes the
algebraic verification.

\section*{Acknowledgements}
This work was supported by the Guangdong Basic and Applied Basic Research Foundation (No.\ 2022A1515110967 and No.\ 2023A1515011809). The author acknowledges the use of AI-assisted tools in manuscript preparation and assumes full responsibility for all mathematical statements and proofs.

\end{document}